\documentclass[11pt]{article}
\usepackage{amssymb}
\usepackage{amsthm}
\usepackage{amsmath,amsfonts,amssymb}
\usepackage{graphicx}
\usepackage{algpseudocode}
\usepackage{algorithm}
\usepackage{natbib}

\newtheorem{theorem}{Theorem}[section]

\newtheorem{definition}[theorem]{Definition}
\newtheorem{remark}[theorem]{Remark}

\def\sqr#1#2{\vbox{\hrule height .#2pt
\hbox{\vrule width .#2pt height #1pt \kern #1pt
\vrule width .#2pt}\hrule height .#2pt }}

\def\begi{\begin{itemize}}
\def\endi{\end{itemize}}
\def\bega{\begin{array}}
\def\enda{\end{array}}

\def\forall{\hbox{for every~ }}

\def\forall{\hbox{for all}~}

\def\bel{\begin{equation}\label}
\def\eeq{\end{equation}}

\begin{document}

\title{Dynamic Congestion and Tolls with Mobile Source Emission}
\author{Ke Han$^{a}\thanks{e-mail: kxh323@psu.edu;}$
\qquad Terry L. Friesz$^{b}\thanks{e-mail: tfriesz@psu.edu;}$
\qquad Hongcheng Liu$^{b}\thanks{e-mail: hql5143@psu.edu;}$
\qquad Tao Yao$^{b}\thanks{e-mail: tyy1@engr.psu.edu;}$ \\\\
$^{a}$\textit{Department of  Mathematics}\\
\textit{Pennsylvania State University, PA 16802, USA}\\
$^{b}$\textit{Department of  Industrial and Manufacturing Engineering,}\\
\textit{Pennsylvania State University, PA 16802, USA}}

\date{}
\maketitle

\begin{abstract}
This paper proposes a dynamic congestion pricing model that takes into account mobile source emissions. We consider a tollable vehicular network where the users selfishly minimize their own travel costs, including travel time, early/late arrival penalties and tolls. On top of that, we assume that part of the network can be tolled by a central authority, whose objective is to minimize both total travel costs of road users and total emission on a network-wide level. The model is formulated as a {\it mathematical program with equilibrium constraints} (MPEC) problem and then reformulated as a {\it mathematical program with complementarity constraints} (MPCC).  The MPCC is solved using a quadratic penalty-based gradient projection algorithm.  A numerical study on a toy network illustrates the effectiveness of the tolling strategy and reveals a Braess-type paradox in the context of traffic-derived emission. 
\end{abstract}

\section{\label{Intro}Introduction}

According to the US  Environmental Protection Agency (2006), in 2003, the transportation sector contributed to 27 percent of total U.S. greenhouse gas (GHG) emissions. This number is expected to grow rapidly with an estimated increase of transportation energy use by 48 percent by 2015. Future transportation service network designs ought to take into account environmental issues.

In this paper, we propose a dynamic second-best congestion toll problem with embedded emission model for the management and control of tollable vehicular networks. We assume that users of a given network are selfishly minimizing their own disutility, which consists of travel delay,  early/late arrival penalties as well as the price of tolls.  On top of that, there exists a central authority that undertakes the role of the Stackelberg leader, whose objective includes two different aspects: the network efficiency and the environmental well-being in the presence of vehicle-driven emission.

The upper-level decision variable for the central authority (Stackelberg leader) is a dynamic congestion toll imposed on certain links of the network; while the lower-level decision variables for the travelers (Stackelberg follower) include route and departure time choices. The proposed congestion pricing problem with embedded emission model is formulated as a {\it mathematical programming with equilibrium constraints} (MPEC) problem, with multiple objectives including the mitigation of both congestion and traffic emission on a network-wide level.

To solve the multi-objective MPEC problem, we start by rewriting the {\it
differential variational inequality} (DVI) formulation of dynamic user equilibrium into a differential complementarity problem. Then, we employ a weighted-sum scalarization method to handle the multiple objectives. With these two steps, the multi-objective MPEC problem is transformed into a single-objective {\it mathematical program with complementarity constraints} (MPCC). To avoid the loss of constraint qualification, we relax the mathematical program by applying a quadratic penalty-based method. The relaxed problem is then solved with a  gradient projection method mentioned  in \cite{DODG}.

\subsection{Congestion toll pricing}
The idea of employing toll pricing to mitigate  congestion
arises from the congestion pricing strategy originally proposed by \cite{Pigou}. In the literature,
toll pricing problems can be classified into two categories: 1) first-best
toll pricing, which means every arc of the network is tollable; and 2) second-best toll
pricing, which assumes that only a subset of arcs is tolled for political
or other reasons. Examples of the first category include marginal social cost pricing
strategy \citep{AK}, and several other models and methodologies \citep{HR, Dial1999, Dial2000}. Regarding the second-best tolling strategy, \cite{LH} propose a {\it mathematical program with equilibrium constraints} (MPEC) approach to compute the optimal toll prices. All the aforementioned literature are restricted to the static case. For a comprehensive review on static road pricing problems, the reader is referred to \cite{YH}. By nature of these problems, only route choices of travelers are captured by the models.

In the past two decades, {\it dynamic traffic assignment} (DTA) models and dynamic congestion tolling problems have received increased attention due to their capability of capturing not only route choices but also  departure time choices of travelers.  Dynamic congestion pricing in the presence of traffic bottlenecks  are investigated in   \cite{ADL, AK, Braid, DL}. 
 \cite{Friesz2007} propose an MPEC problem and a solution approach to determine
the optimal second-best tolling strategy, using the {\it link delay model} (LDM) original introduced by \cite{Friesz1993}. \cite{Yao} further study a dynamic congestion pricing problem
in the presence of demand uncertainty. \cite{Wismans} employs the cell transmission model to study multi-objective congestion management problem. He uses a genetic algorithm and response surface methods for solving the MPEC problems.  A more complete review on existing dynamic congestion pricing models and solution approaches is presented in \cite{Yao}.

 In this article, we seek to explore the effectiveness of second-best tolling strategies in minimizing both traffic congestion and automobile-induced emissions. To this end, we propose
a multi-objective MPEC problem  to determine the optimal toll price. Such an MPEC model has a lower-level dynamic user equilibrium problem that employs the LWR-Lax model \citep{Friesz2013} for the {\it dynamic network loading} (DNL) subproblem. The contribution made by this paper is as follows. 
\begin{itemize}
\item We propose an approach of embedding emission models into the {\it dynamic network loading} (DNL) submodel of the dynamic user equilibrium problem. Such an approach is compatible with a variety of traffic flow models and emission models, which may capture vehicle spillback, and acceleration/deceleration. 
\item We propose to reformulate the dynamic MPEC model into a single-level optimal control problem using the equivalence between the DVI and the complementarity systems. The reformulation admits existing solution schemes.
\item A Braess-type paradox is reported in our numerical results which extends the classical Braess paradox \citep{Braess} to a dynamic case and to the context of environmental well-being. Such observation delivers further managerial insights to sustainable road network management. 
\end{itemize}

\subsection{Dynamic user equilibrium model} 
In this article, we employ the {\it simultaneous route-and-departure choice} (SRDC) dynamic user equilibrium model proposed by \cite{Friesz1993}. For the SRDC notion of DUE, unit travel cost, including early and late arrival penalties, is identical for those route and departure time choices selected by travelers between a given origin-destination pair. Such problem is articulated and formulated as a {\it variational inequality} (VI) in \cite{Friesz1993}. The DUE model typically consists of two major components: the mathematical notion of equilibrium among Nash agents, and the network performance model known as the  dynamic network loading (DNL) submodel. The DNL aims at describing and predicting temporal evolution of system
states by combining link dynamics and flow propagation constraints with link and path delay models.  Note that, by referring to the network loading procedure, we are neither
employing nor suggesting a sequential approach to the study and computation
of DUE. Rather, a subset of the equations and inequalities comprising a
complete DUE model may be grouped in a way that identifies a traffic
assignment subproblem and a network loading subproblem. Such a grouping and
choice of names is merely a matter of convenient language that avoids
repetitive reference to the same mathematical expressions. Use of such
language does not alter the need to solve both the assignment and loading
problems consistently and, thus, simultaneously.

\cite{Friesz2001} solve the differential variational inequality (DVI) formulation of DUE and the DNL subproblem simultaneously by formulating the arc dynamics, flow propagation constraints as a system of ordinary differential equations with state-dependent time lags. By doing so, they turn the DUE problem into a ``single-level" DVI problem that can be handled in the optimal control framework.  In addition, necessary conditions for optimal control problems with state-dependent time lags are derived therein. In \cite{FrieszMookherjee}, the theory of optimal control and the theory of infinite dimensional VIs are combined to create an implicit fixed point algorithm for calculating DUE.  \cite{Friesz2011} extend the time scale in which DUE problems are analyzed from within-day to day-to-day.  A dual-time scale DUE are articulated and solved as a result. \cite{Friesz2013} consider the Lighthill-Whitham-Richards model \citep{LW, Richards} for the DNL submodel.  The authors employ a variational method, known as Lax formula \citep{Lax, Evans}, derived for scalar conservation laws and Hamilton-Jacobi equations. In that paper, the DNL subproblem is formulated as a system of {\it differential algebraic equations} (DAEs), which can be efficiently solved for medium- and large-scale networks.

\subsection{Automobile emission models} 

Modeling approaches for automobile source emission  can be classified into three categories: microscopic, macroscopic and mesoscopic approaches. The microscopic emission models are relatively accurate: they characterize the emission rate on the level of a single vehicle, based on the physical attributes of the vehicle, driving behavior of the driver, as well as the surrounding environment.  It is assumed that the emission rate $e(t)$ of a moving vehicle is be expressed as a function of instantaneous velocity $v(t)$ and acceleration $a(t)$, 
\begin{equation}\label{micro}
e(t)~=~f_1\big(v(t),\,a(t)\big)
\end{equation}
such models can be easily calibrated and validated in a laboratorial environment.  There are several emission models based on the microscopic emission mechanism, such as \cite{Barth, Panis} and \cite{Rakha}. 
The drawback of the microscopic modeling approach is the lack of measurements associated with each individual car on the road. On the other hand, it is relatively easy to measure the traffic dynamics on a macroscopic level. The  macroscopic emission models (Ekstr\"om et al. 2004) express the average emission rate $\bar e(t)$ on a road segment as a function of the average density  $\bar\rho$ and average velocity $\bar v(t)$ in that same segment
\begin{equation}\label{macro}
\bar e(t)~=~f_2\big(\bar \rho(t),\,\bar v(t)\big)
\end{equation}
The drawback of the macroscopic modeling approaches for emission lies in the fact that the model is difficult to calibrate and validate, due to insufficient emission measurements on a road. The third type of emission models, the mesoscopic emission models, approximate individual vehicles' dynamics using macroscopic flow models and measurements. Then the macroscopic emission rate is aggregated among individual vehicles, while the emission rate of each individual vehicle is computed at a microscopic level. The mesoscopic models (Csik\'os et al. 2011, Csik\'os and Varga 2011,  Zegeye et al. 2010) take the modeling advantages of both macroscopic traffic flow models and microscopic emission models, avoiding the drawbacks of the previous two approaches.  However, combining a macroscopic traffic model which ignores granularity of microscopic quantities with an accurate microscopic emission models may introduce additional uncertainties to the model. Therefore, the mesoscopic models need to be carefully calibrated and validated using macroscopic traffic and emission measurements.

In this paper, the process of emission estimation is be embedded in the procedure of  dynamic network loading within the DUE problem. The DNL  procedure also provides a basis for the comparison of various microscopic and macroscopic emission functions, among which we distinguish between the two-argument functions $e(t)=f_1\big(v(t),\,a(t)\big)$ and the single-argument functions $e(t)=f_3\big(v(t)\big)$. 

The two-argument functions, such as the one proposed in the modal emission model  \citep{Barth}, apply a physical approach that matches the power demand  of a vehicle to various driving conditions including: low/high speed cruising, acceleration/deceleration, idling, and  stop-and-go, etc..  Such models are relatively accurate, and can be calibrated for different types of vehicles. However, it is relatively difficult to integrate the modal model into a macroscopic traffic flow model. In particular, the higher order traffic quantities such as acceleration/deceleration cannot be sufficiently captured by first-order models such as the Lighthill-Whitham-Richards conservation law model. We will have more to say about this in Section \ref{sectwoarg}. 

 On the other hand, the one-argument emission functions typically depends on the average speed. \cite{Rose} show that when traveling speed is under $80$ (km/hour), the relation between speed $v$ (in km/hour) and $HC/CO$ emissions $e_x$ (in pound/km) can be approximated by (for now and sequel,  $e_x$ denote the emission per unit distance).
\begin{equation}\label{rose}
e_x~=~b_1\,v^{-b_2}
\end{equation}
where $b_1,\,b_2$ are parameters depending on vehicle type and surrounding environment. \cite{KM} collected driving pattern data in Sydney and found that $NO_y$ emission $\tilde e_x$ can be modeled by
\begin{equation}\label{kentmudford}
\tilde e_x~=~\tilde b_1+{\tilde b_2\over v}
\end{equation}
According to the  Emission Factor Model 2000 \citep{emfac} by California Air Resources Board, constantly updated since 1988, the hot running emissions per unit distance 
\begin{equation}\label{emfac}
\hat e_x~=~\hbox{BER}\times \exp\left\{ \hat b_1(v-17.03)+\hat b_2(v-17.03)^2 \right\}
\end{equation}
where BER stands for basic emission rates, which are constants associated with $CO,\,NO_y,\,HC$. The unit of velocity is in mile/hour, the unit of $\hat e_x$ is in gram/mile.

\subsection{Solving MPEC problems}
The mathematical program with equilibrium constraints (MPEC), by its bi-level and non-convex nature,   often creates computational difficulties. A common approach to solve an MPEC problem is to reformulate the bi-level program into a {\it mathematical program with complementarity constraints} (MPCC), see \cite{Ban} and \cite{DODG}.  However, as noted in \cite{RM}  and in \cite{Ban}, the complementarity constraints might lose certain constraint qualifications. To resolve this issue, some regularization techniques are proposed in the literature.  \cite{RW}  study a relaxation approach, which is then applied by \cite{Ban}  to solve a continuous network design problem. \cite{Anitescu}  proposes an ${l}_1$-penalty approach and studies its impact on the convergence of an interior point algorithm; while \cite{MM} test a quadratic penalty function. According to their numerical results, quadratic penalty is a promising approach to handle complementarity constraints. However, all of the discussions above focus on  MPCC or MPEC in the context of finite dimensional programs. Regarding  continuous-time dynamic MPECs, numerical techniques were scarcely visited. Existing literature on continuous-time MPECs includes  the single-level reformulation proposed in \cite{Friesz2007}, the metaheuristic approach by  \cite{Yao} and a simultaneous discretization-based method in \cite{Raghunathan}

This paper utilizes the quadratic penalty method to solve the proposed dynamic MPEC problem. In particular, we will drop the complementarity constraints from the MPCC reformulation, and attach to the objective function a quadratic penalty function for the dropped constraints. The numerical results show general solvability and effectiveness of the proposed numerical method.

\subsection{Organization} 
The rest of this article is organized as follows. Section \ref{secDUE} recaps the dynamic user equilibrium model and its reformulation as variational inequality and differential variational inequality. We also present the dynamic network loading (DNL) submodel employed in this paper. In Section \ref{secemissionmodels}, two emission models are discussed in detail and embedded in the DNL subproblem. In Section \ref{secMultiobj} and Section \ref{secPPGM}, we  present the multi-objective MPECs, MPCCs and discuss solution methods based on a gradient projection method with quadratic penalty for the complementarity constraints. In Section \ref{secnumerical}, a sustainable congestion toll problem on a toy network is solved using techniques mentioned before. The optimal toll is meant to optimize two objectives simultaneously, under equilibrium flow. In particular, the numerical results demonstrate the effectiveness of our proposed methodology in reducing both emission and congestion levels. A Braess-type paradox is also observed in connection with these two objectives.

\section{Dynamic User Equilibrium}\label{secDUE}
In this section, we briefly review the DUE problem which serves as the lower-level component of our MPEC formulation.  The  DUE model is formulated as a variational inequality in \cite{Friesz1993} and then as a differential variational inequality in \cite{FrieszMookherjee}, then solved via a fixed-point algorithm in Hilbert space by \cite{Friesz2011}.

\subsection{The DUE formulation}

Let us consider a fixed planning horizon $[t_0,\, t_f]\subset \Re$. The most crucial ingredient of a dynamic user equilibrium model is the path delay operator, which provides travel delay along 
a path $p$ per unit of flow departing from the origin of that path; it is denoted by
\begin{equation}
D_{p}(t,h)\qquad \forall p\in \mathcal{P}
\end{equation}
where $\mathcal{P}$ is the set of paths employed by travelers, $t$
denotes departure time, and $h$ is a vector of departure rates. The path delay operators usually do not take on any closed form, instead they can only be evaluated numerically through the dynamic network loading (DNL) procedure. From these we construct effective unit path delay operators
\begin{equation}\label{psidef}
\Psi _{p}(t,h)=D_{p}(t,h)+F\left[ t+D_{p}(t,h)-T_{A}\right] \qquad   \forall p\in \mathcal{P}
\end{equation}%
where $T_{A}$ is the desired arrival time. We introduce the fixed trip matrix $\big(Q_{ij}: (i,\,j)\in\mathcal{W}\big)$, where each $Q_{ij}\in \Re _{+}$ is the fixed travel demand, expressed
as a volume, between origin-destination pair $\left( i,j\right) \in \mathcal{%
W}$ and $\mathcal{W}$ is the set of all origin-destination pairs.
Additionally, we define the set $\mathcal{P}_{ij}$ to be the subset of
paths that connect origin-destination pair $\left( i,j\right) \in \mathcal{W}
$.

We write the flow conservation constraints as
\begin{equation}\label{cons}
\sum_{p\in P_{ij}}\int_{t_0}^{t_f}h_p(t)\,dt~=~Q_{ij}\qquad\forall (i,\,j)\in\mathcal{W}
\end{equation}

\noindent Let us denote the vector of path flows by $h=\{h_p: p\in\mathcal{P}\}$,  in addition, we stipulate that the path flows are square integrable:
$$
h\in\big(\mathcal{L}_+^2[t_0,\,t_f]\big)^{|\mathcal{P}|}
$$
where $\big(\mathcal{L}_+^2[t_0,\,t_f]\big)^{|\mathcal{P}|}$ denotes the positive cone of the $|\mathcal{P}|$-fold product of the space $\mathcal{L}^2[t_0,\,t_f]$ consisting of square-integrable functions on $[t_0,\,t_f]$. Using the notation and concepts we have mentioned, the feasible region for path flows is
\begin{equation}\label{feasible}
\Lambda _{0}=\left\{ h\geq 0:\sum_{p\in \mathcal{P}_{ij}}%
\int_{t_{0}}^{t_{f}}h_{p}\left( t\right) dt=Q_{ij}\text{ \ \ \ }\forall
\left( i,j\right) \in \mathcal{W}\right\} \subseteq \left( \mathcal{L}_{+}^{2}\left[
t_{0},t_{f}\right] \right) ^{\left\vert \mathcal{P}\right\vert }
\end{equation}
Let us also define the essential infimum of effective travel delays%
\begin{equation*}
v_{ij}=\hbox{essinf} \left[ \Psi _{p}(t,h):p\in \mathcal{P}_{ij}\right] \text{ \ \
\ \ }\forall \left( i,j\right) \in \mathcal{W}
\end{equation*}%
The following definition of dynamic user equilibrium was first articulated by Friesz et al. (1993).
\begin{definition}
\label{duedef}{\bf (Dynamic user equilibrium)}. A vector of departure rates (path
flows) $h^{\ast }\in \Lambda _{0}$ is a dynamic user equilibrium if%
\begin{equation*}
h_{p}^{\ast }\left( t\right) >0,p\in \mathcal{P}_{ij}\Longrightarrow \Psi _{p}\left[
t,h^{\ast }\left( t\right) \right] =v_{ij}
\end{equation*}%
We denote this equilibrium by $DUE\left( \Psi ,\Lambda _{0},\left[
t_{0},t_{f}\right] \right) $.
\end{definition}
Using measure theoretic arguments, Friesz et al. (1993) established that
a dynamic user equilibrium is equivalent to the following variational
inequality under suitable regularity conditions:
\begin{equation}
\left. 
\begin{array}{c}
\text{find }h^{\ast }\in \Lambda _{0}\text{ such that} \\ 
\sum\limits_{p\in \mathcal{P}}\displaystyle \int\nolimits_{t_{0}}^{t_{f}}\Psi _{p}(t,h^{\ast
})(h_{p}-h_{p}^{\ast })dt\geq 0 \\ 
\forall h\in \Lambda _{0}%
\end{array}%
\right\} VI(\Psi ,\Lambda_0 ,\left[ t_{0},t_{f}\right] )  \label{duevi}
\end{equation}%

It has been noted in Friesz et al. (2011) that (\ref{duevi}) is equivalent
to a differential variational inequality. This is most easily seen by noting
that the flow conservation constraints may be re-stated as a two-point boundary value problem:
\begin{equation*}
\left. 
\begin{array}{l}
\displaystyle \frac{dy_{ij}}{dt}=\displaystyle \sum\limits_{p\in \mathcal{P}_{ij}}h_{p}\left( t\right) 
 \\ 
y_{ij}(t_{0})=0\\ 
y_{ij}\left( t_{f}\right) =Q_{ij}
\end{array}
\right\}\qquad \forall \left( i,j\right) \in \mathcal{W}
\end{equation*}%
where $y_{ij}(\cdot)$ is interpreted as the cumulative traffic  that has departed between origin-destination pair $(i,\,j)\in\mathcal{W}$. As a consequence, (\ref{duevi}) may be expressed as the following differential variational
inequality (DVI):\ \ \ \ \ \ \ \ \ \ \ \ \ \ \ \ \ \ \ \ \ \ \ \ \ \ \ \ \ \
\ \ \ 
\begin{equation}
\left. 
\begin{array}{c}
\text{find }h^{\ast }\in \Lambda \text{ such that} \\ 
\displaystyle\sum_{p\in \mathcal{P}}\displaystyle\int\nolimits_{t_{0}}^{t_{f}}\Psi _{p}(t,h^{\ast
})(h_{p}-h_{p}^{\ast })dt\geq 0 \\ 
\forall h\in \Lambda
\end{array}
\right\} DVI(\Psi ,\Lambda,\, [t_{0},\,t_{f}])  \label{dvi0}
\end{equation}

\noindent where
\begin{equation}\label{lambdadef}
\Lambda ~=~\left\{ h\geq 0:\frac{dy_{ij}}{dt}=\sum_{p\in P_{ij}}h_{p}\left(
t\right) ,\text{\ }y_{ij}(t_{0})=0,\text{\ }y_{ij}\left( t_{f}\right) =Q_{ij}%
\text{ \ \ }\forall \left( i,j\right) \in \mathcal{W}\right\}
\end{equation}

\noindent Analysis and computation of dynamic user equilibrium is tremendously simplified by stating it as a differential variational inequality (DVI), due to the optimal control framework inherent in the DVI problems. Finally, we are in a position to state a result that permits the solution of the DVI (\ref{dvi0}) to be obtained by solving a fixed point problem:
\begin{theorem}
\label{fix}{\bf(Fixed point re-statement)}. Assume
that $\Psi _{p}(\cdot ,h):\left[ t_0,\,t_f \right] \longrightarrow \Re_{+}$ is measurable for
all $p\in \mathcal{P}$, $h\in \Lambda  $. Then the fixed point problem%
\begin{equation}
h~=~P_{\Lambda}\left[ h-\alpha \Psi \left(
t,h\right) \right] \text{,}  \label{duefpf}
\end{equation}%
is equivalent to $DVI(\Psi ,\Lambda,\,[t_0,\,t_f] )$ where $P_{\Lambda
}\left[ \cdot \right] $ is the minimum norm projection
onto $\Lambda$ and $\alpha \in \Re _{+}$.
\end{theorem}
\begin{proof}
See \cite{Friesz2011}.
\end{proof}

\noindent Theorem \ref{fix} suggests a way of solving the dynamic user equilibrium problem via an iterative scheme of the form
$$
h^{k+1}~=~P_{\Lambda}\left[h^k-\alpha\,\Psi(t,\,h^k)\right]
$$
where $h^{k+1},\,h^k\in\Lambda$ are two consecutive iterates. Convergence of such scheme requires monotonicity, or a weaker notion of monotonicity, of the effective delay operator $\Psi(t,\,\cdot)$, which is discussed in \cite{Nagurney} and \cite{Friesz2011}. 

\subsection{The DNL subproblem}\label{secDNL}

A crucial component of the VI and DVI formulations of dynamic user equilibrium is the effective delay operator, typically obtained from {\it dynamic network loading} (DNL), which is a subproblem of a complete DUE model. Any DNL must be consistent with the established path flows and link/path delay models, and is usually performed under the {\it first-in-first-out} (FIFO) principle.

In this paper, we employ the LWR-Lax model proposed by \cite{Friesz2013}. The LWR-Lax model is a simplified version of the LWR model on networks. It is based on the assumption that any queues induced by congestion does not have physical size, thus no spill back occurs in the network. The link dynamics, link delay models and route and departure time choices are expressible as a system of {\it differential algebraic equations} (DAEs). The  DAE system for the network loading submodel is derived via a variational method, known as the Lax-Hopf formula \cite{Evans, Lax}, for scalar conservation laws and Hamilton-Jacobi equations.  Due to space limitation, we will present such DAE system below without elaborating its mathematical details. The reader is referred to \cite{Friesz2013} for more discussion. However, it is important for us to note that the modeling framework and solution methodology for sustainable congestion management proposed in this paper is independent of the DNL model chosen. In other words, it is expected that our model should yield similar qualitative result and managerial insights when other types of DNL models are employed in the computation of DUE and in the estimation of network-wide emission.

Given a vehicular network represented as a directed graph $G(\mathcal{A},\,\mathcal{V})$, where $\mathcal{A}$ denotes the set of arcs (links), and $\mathcal{V}$ denotes the set of vertices (nodes). We define for each arc $e\in\mathcal{A}$, the free flow speed $v^e_0$ and the jam density $\rho_{jam}^e$. Assume that the arc dynamic is governed by the following conservation law
\begin{equation}\label{claw}
\partial_t\,\rho^e(t,\,x)+\partial_x\,f^e\big(\rho^e(t,\,x)\big)~=~0\qquad (t,\,x)\in[t_0,\,t_f]\times [0,\,L^e]
\end{equation}
where $[t_0,\,t_f]\times[0,\,L^e]$ denotes the temporal-spatial domain of the partial differential equation. $\rho^e(t,\,x)$ represents the (local) vehicle density at location $x$ and time $t$. The fundamental diagram $f^e(\cdot)$, as a function of local density only,  is assumed to be continuous, concave and vanishes at $\rho^e=0$ and $\rho^e=\rho_{jam}^e$, where $\rho_{jam}^e$ represents jam density of link $e$. Let us  introduce a few more notations:
\begin{align*}
&\mathcal{W}:~~\hbox{the set of origin-destination pairs in the network}\\
&\mathcal{P}:~~\hbox{the set of paths utilized by travelers}\\
&\mathcal{P}_{ij}:~~\hbox{the set of utilized paths that
connects origin-destination pair \({(i,j)\in \mathcal{W}}\)}\\
&p~=~\{e_1,\,e_2,\,\ldots,\,e_{m(p)}\}\in\mathcal{P},\,\,e_i\in\mathcal{A} :~~\hbox{path represented by the set of  arcs it traverses, where }\\
&\qquad ~~ m(p) \hbox{ denotes the number of links traversed by path } p\\
&h_p(t):~~\hbox{departure rate (path flow) at origin, associated with path}~p\\
&q^e_p(t):~~\hbox{link}~e~ \hbox{entering flow associated with path}~p\\
&w^e_p(t):~~\hbox{link}~e~\hbox{exiting flow associated with path}~p\\
&Q^e_p(t):~~\hbox{cumulative entering vehicle count on arc}~e~\hbox{associated with path}~p\\
&W^e_{p}(t):~~\hbox{cumulative exiting vehicle count at arc}~e~\hbox{associated with path}~p\\
&L^e:~~\hbox{length of arc}~e\in\mathcal{A} 
\end{align*}

\noindent In addition, let us define the following function 
$$
\phi^e(u)~=~\min \left\{\rho\in [0,\,\rho^e_{jam}]: ~~ f^e(\rho)~=~u\right\} \qquad u\in[0,\,M^e]
$$
and its Legendre transformation
$$
\psi^e(p)~=~\sup_u\left\{up- \phi^e(u)\right\}
$$
where $M^e$ denotes the flow capacity of link $e$. Moreover, we denote by $D(t\,;\,Q^e)$ the time taken to traverse link $e$ when the time of entry is $t$, under the link entering flow profile $Q^e$ where 
$$
Q^e(t)~\doteq~\sum_{e\in p}Q_p^e(t),\qquad \qquad W^e(t)~\doteq~\sum_{e\in p}W_p^e(t)
$$

\noindent By convention, we write~ $q_p^{e_1}(t)~=~h_p(t)$,~ $w_p^{e_0}(t)~=~h_p(t)$. The following DAE system (\ref{def1})-(\ref{divrt}) for the dynamic network loading is given in \cite{Friesz2013}.
\begin{align}
\label{def1}
&Q^{e}(t)~\doteq~\sum\limits_{e\in p}Q^{e}_{p}(t),\quad q^{e}(t)~\doteq~\sum_{e\in p}q^{e}_p(t),\quad w^e(t)~\doteq~\sum_{e\in p}w^e_p(t)\\
\label{def2}
 &{d\over dt}Q^e_p(t)~=~q^e_p(t),\quad {d\over dt}W^e(t)~=~w^e(t)\quad \forall~p\in\mathcal{P}\\
\label{junc}& q_p^{e_{i}}(t)~=~w^{e_{i-1}}_p(t);\qquad i\in[1,\,m(p)],~p\in\mathcal{P}\\
\label{Lax'}&\displaystyle W^e(t)~=~\min\limits_{\tau}\Big\{Q^e(\tau)+L^e\psi^e\Big({t-\tau\over L^e}\Big)\Big\};\qquad \forall~e\in\mathcal{A} \\
\label{fconstr}&Q^e(t)~=~W^e\big(t+D(t;\,Q^e)\big);\\
\label{divrt}&\displaystyle w^{e_i}_{p}\big(t+D(t;\,Q^{e_i})\big)~=~{q^{e_{i}}_p(t)\over q^{e_{i}}(t)}w^{e_i}\big(t+D(t;\,Q^{e_i})\big);\qquad i\in[1,\,m(p)],~p\in\mathcal{P}
\end{align}

We note that (\ref{def1}) is definitional, i.e. the traffic on an arc is disaggregated by different route choices.  (\ref{junc}) represents the fundamental recursion, which allows the algorithm to carry forward to the next arc in the path. (\ref{Lax'}) is the Lax-Hopf formula (Bressan and Han, 2011a,b).    (\ref{fconstr}) is often referred to as the flow propagation constraint, from which the travel time function $D(\cdot; Q^e)$ can be solved. (\ref{divrt}) describes the model of diverge junctions where travelers' route choices are explicitly considered.

One  shortcoming of the above DNL procedure is the lack of consideration for spillback. Vehicle spillback  not only aggravates congestion and causes higher travel delay, but also produce more stop-and-go waves (Colombo and Groli 2003) that affects the estimation of traffic emission. However, as mentioned before, our modeling framework can subsume any network loading procedures regardless of the link dynamic, flow propagation and delay model employed. One aspect of future research is to incorporate vehicle spillback in the DNL submodel and investigate its influence on the best tolling strategy and overall performance of the traffic network in terms of travel delay and environmental impact.

\section{The  DNL Submodel Integrated with Emission Models}\label{secemissionmodels}
This section, presents two  approaches for modeling  traffic emission on a road network. The emission model will be considered in connection with the DNL subproblem. As a result, the output of the DNL subproblem will include 1) the effective delay associated with each pair of departure time and route choices, and 2) the emission associated with each pair of departure time and route choices, as well as the total emission of the network.

\subsection{Emission as a functional of velocity and acceleration}\label{sectwoarg}

Consider a road network $G(\mathcal{A},\,\mathcal{V})$.  For each arc $a\in\mathcal{A}$, let us denote by $\rho_a(t,\,x),\, v_a(t,\,x)$ the local density and average velocity of vehicles at time $t$ and location $x$. The classical Lighthill-Whitham-Richards (LWR) model (Lighthill and Whitham 1955, Richards 1956) describes the temporal-spatial evolution of $\rho_a(t,\,x)$ via the following scalar conservation law
\begin{equation}\label{lwrpde}
{\partial\over \partial t}\rho_a(t,\,x)+{\partial\over\partial x}\Big(\rho_a(t,\,x)\, v\big(\rho_a(t,\,x)\big)\Big)~=~0
\end{equation}
where the velocity is expressed as an explicit function of density. The map $\rho\mapsto \rho\cdot v(\rho)$ is interpreted as the fundamental diagram.

Following emission models proposed by \cite{Barth, Smit2006}, we assume that the emission rate $e(t)$ of a moving vehicle can be modeled as a function of its instantaneous velocity $v(t)$ and acceleration $a(t)$: 
\begin{equation}\label{efunc}
e(t)~=~\mathcal{E}\big(v(t),\, a(t)\big)\footnote{We note that the function $\mathcal{E}$ should not be taken directly from a microscopic emission model, such as that in \cite{Barth}. Rather, such a function should be carefully calibrated and validated in connection with macroscopic traffic models and the result should reflect emission rate on a macroscopic level.}
\end{equation}

\noindent Consider an arc \,$a\in\mathcal{A}$\, expressed as a spatial interval $[0,\,L_a]$ and the weak solution $\rho_a(t,\,x)$, $(t,\,x)\in[t_0,\,t_f]\times [0,\,L_a]$ of the LWR conservation law (\ref{lwrpde}).  Then the total emission on this arc is computed as 
\begin{align}\label{totalemission}
&\int_{t_0}^{t_f}\int_{0}^{L_a} \rho_a(t,\,x)\cdot e(t,\,x)\,dx\,dt\\
~=~& \int_{t_0}^{t_f}\int_0^{L_a}\rho_a(t,\,x)\left(\mathcal{E}\left(v_a(t,\,x),\, {D\over Dt}\,v_a(t,\,x)\right)\right)\,dx\,dt\\
\label{vamodel}
~=~& \int_{t_0}^{t_f}\int_0^{L_a}\rho_a(t,\,x)\left(\mathcal{E}\left(v_a,\,\, {\partial\over \partial t}v_a+v_a\cdot {\partial\over\partial x} v_a\right)\right)\,dx\,dt
\end{align}
where  ${D\over Dt}\doteq {\partial\over\partial t} + v_a\cdot {\partial\over\partial x}$ is the material derivative in Eulerian coordinates corresponding to the acceleration of the car in Lagrangian ones. The variable $e(t,\,x)$ denotes the local emission rate at location $x$ at time $t$. 

Notice that the expression in (\ref{vamodel}) is not well-defined in the context of scalar conservation laws as the solution $\rho_a$ and $v_a$ are not continuous in general. As an alternative, one may interpret \eqref{vamodel} in a discrete-time framework such as {\it cell transmission model} (CTM) proposed in \cite{CTM1, CTM2}. The partial derivatives are approximated by finite-differences and the integrals are approximated by appropriate quadratures. Again, the function $\mathcal{E}$ should be calibrated in connection with cell transmission model. The implementation of the above emission model is straightforward, but is not within the scope of this paper.

\subsection{Emission as a functional of velocity}\label{secemissionmodel}

The second emission model discussed in this section is a speed-related emission models. Such model can be easily embedded into the dynamic network loading subproblem mentioned in Section \ref{secDNL}. Within this model, it is assumed that the average emission rate of a traveling vehicle is expressed as a function of its average travel speed for an arbitrary period of time. Such a model ignores  granularity related to instantaneous speed and acceleration/deceleration and is calibrated and validated via empirical data, see, for example, Rose et al. (1965), Kent and Mudfor (1979) and CARB (2000). The function relating average emission rate to average velocity is written as:
\begin{equation}\label{efuncvel}
\bar e(t)~=~\overline{\mathcal{E}}\big(\bar v(t)\big)
\end{equation}
where $\bar e(t)$ and $\bar v(t)$ denotes average emission rate (per unit of time) and average velocity, respectively.   $\bar v(t)$ can be averaged over a time period during which the vehicle traverses a whole link.  Specifically, given any feasible path flows $h\in\Lambda$, one can solve the DNL problem using the DAE system proposed by \cite{Friesz2013}.  Let $D_p(t,\,h)$ be the time taken for a driver who departs at $t$ to traverse the path $p$.  In addition, we let $\tau_{a_i}^p(t)$ be the time of exit from arc $a_i$ given that departure from the origin occurs at  time $t$ and path $p$ is followed, where $p=\{a_1,\,\ldots,\,a_{m(p)}\}$. Then the average speed on link $a_i$ when the departure time from the origin occurs at $t$, denoted by $\bar v_{a_i}(t,\,h)$, is given by
$$
\bar v_{a_i}(t,\,h)~=~{L_{a_i}\over \tau_{a_i}^p(t)-\tau_{a_{i-1}}^p(t)}\qquad t\in[t_0,\,t_f],\quad  p\in\mathcal{P}
$$
where $L_{a_i}$ is the length of arc $a_i\in p$. In view of identity (\ref{efuncvel}),  the contribution to total emission of user departing at time $t$ along path $p$ is given by 
\begin{equation}\label{emissionvel}
E_p(t,\,h)~=~\sum_{a_i\in p}\left(\tau_{a_i}^p(t)-\tau_{a_{i-1}}^p(t)\right)\cdot \overline{\mathcal{E}}\left({L_{a_i}\over \tau_{a_i}^p(t)-\tau_{a_{i-1}}^p(t)}\right)
\end{equation}
\noindent The left hand side of \eqref{emissionvel} is expressed in the form of an operator 
$$
\begin{array}{c}
\displaystyle E: \big(\mathcal{L}^2_+[t_0,\,t_f]\big)^{|\mathcal{P}|}~\rightarrow~  \big(\mathcal{L}^2_{++}[t_0,\,t_f]\big)^{|\mathcal{P}|}
\\\\
h~\mapsto~E(\cdot,\,h)~=~\big(E_p(\cdot,\,h):~p\in\mathcal{P}\big)
\end{array}
$$
\noindent Such an operator depends only on knowledge of the delay operator, and in turn is known completely once the vector of path flows $h$ is given. This concludes our embedding of the emission model into the DNL procedure. Another advantage of expressing the path and departure-time specific emission as an operator is that it facilitates the derivation of gradient of the objective function presented  in Section \ref{secPPGM}.

The total emission in the network given the vector of path flows $h$ is readily calculated as
\begin{equation}\label{totalemission}
\hbox{total emission}~=~\sum_{p\in\mathcal{P}}\int_{t_0}^{t_f} h_p(t)\cdot E_p(t,\,h)\,dt
\end{equation}

\begin{remark}\label{vemission}
The emission functions proposed in Rose et al. (1965), Kent and Mudford (1979) and CARB (2000) all measure the emission per unit distance $e_x$ against travel speed, where $e_x$ denotes the spatial partial derivative of emission rate. See, for example, \eqref{rose}, \eqref{kentmudford} and \eqref{emfac}.  We employ a simple technique to transform the emission per unit distance to the emission per unit time so that the above modeling framework can be applied. Specifically, notice that 
\begin{equation}\label{emissionrate}
\bar e(t)~=~{\partial \over \partial t} e~=~{\partial x\over \partial t}\,{\partial \over \partial x} e(t,\,x)~=~\bar v(t)\cdot e_x
\end{equation}
thus transforming $e_x$ to $\bar e(t)$.
\end{remark}

\section{Multi-objective Toll Pricing}\label{secMultiobj}
Most of the current MPEC-based dynamic traffic assignment problems deal with a single objective. \cite{LH}  study  efficient tolling strategies in a static network.
\cite{Friesz2007} extend their work to consider dynamic congestion tolls.  \cite{Yao} further investigate the dynamic congestion pricing problem
with demand uncertainty. In these abovementioned studies, the Stackelberg leader (central authority) seeks to minimize a single objective function which is the total (effective) delay. However, the problem of  congestion pricing with emission consideration, as we study in this paper, is more subtle. Difficulties and paradoxes may arise from the fact that the most environment-friendly driving conditions turn out to be inefficient in terms of travel time \citep{emfac}. Therefore, one major challenge faced by researchers is to resolve the conflict between two potentially opposing objectives: the transportation efficiency and emission level. In dealing with such difficulty,  we formulate our MPEC problem as a bi-objective program:
\begin{equation}\label{multiobj}
\min_\mathcal{Y} \mathcal{U}~=~\left[\sum_{p\in \mathcal{P}} \int_{t_0}^{t_f}\Psi_p(t,h^{\ast
})\,h^{*}_p(t)\,dt, \quad \sum_{p\in \mathcal{P}}  \int_{t_0}^{t_f}{E}_p(t,h^*)\,h^{*}_p(t)\, dt \right]
\end{equation}
subject to
\begin{equation}\label{viconstraint}
\sum_{p\in P}\int_{t_0}^{t_f}\left(\Psi _{p}(t,h^{\ast
})+\delta_{a,p}\mathcal{Y}_a\right)\left(h^*_{p}-h_{p}\right)\,dt~\leq~0 \qquad\forall h\in\Lambda
\end{equation}
\begin{equation}
h^{*}\in\Lambda
\end{equation}
\begin{equation}
\Lambda~=~\left\{ h\geq 0:\frac{dy_{ij}}{dt}=\sum_{p\in P_{ij}}h_{p}\left(
t\right) ,\text{\ }y_{ij}(t_{0})=0,\text{\ }y_{ij}\left( t_{f}\right) =Q_{ij}%
\text{ \ \ }\forall \left( i,j\right) \in \mathcal{W}\right\}
\end{equation}
\begin{equation}\label{yulbound}
0~\leq~\mathcal{Y}_a~\leq ~Y_{UB}\qquad \forall a\in\mathcal{A}
\end{equation}

\noindent where   $\mathcal{Y}=\left(\mathcal{Y}_a:a\in\mathcal{A}\right)$. In the objective function defined in \eqref{multiobj}, the effective path delay operator $\Psi_p(\cdot,\,\cdot)$ is defined in \eqref{psidef}, while the ``path emission operator" $E_p(\cdot,\,\cdot)$ is defined in \eqref{emissionvel}.  The first term appearing on the right hand side of (\ref{multiobj}) is the total effective delay; the second term is the total emission on the network.  In Constraint \eqref{viconstraint}, we define
\begin{displaymath}\delta_{a,p}=\begin{cases}1 & \text{if path}~p~\text{traverses arc}~a \\
0 & \text{otherwise} \\
\end{cases}\end{displaymath}
The constant $Y_{UB}\in\Re_{++}$ denotes the prescribed upper bound of the toll.   Constraint (\ref{viconstraint}) is recognized as the variational inequality formulation of DUE, taking the toll prices $\mathcal{Y}_a$ as part of the users' disutility. One crucial component in the formulation above is the effective delay operator $\Psi(\cdot,\,\cdot): \Lambda\rightarrow \big(\mathcal{L}_+^2[t_0,\,t_f]\big)^{|\mathcal{P}|}$. Typically, $\Psi$ is not knowable in closed form; the numerical evaluation of such operator is performed by the dynamic network loading procedure, in particular, by  the DAE system \eqref{def1}-\eqref{divrt} presented in Section \ref{secDNL}.

In summary, the proposed MPEC model for sustainable congestion pricing problem consists of \eqref{multiobj}-\eqref{yulbound},  (\ref{def1})-(\ref{divrt}), and (\ref{emissionvel})-(\ref{totalemission}).

\vspace{0.2 in}

Notice that  \(\mathcal{U}\) is a vector of two objective functions. Thus, minimizing \(\mathcal U\) in (\ref{multiobj}) means that we are seeking to find a Pareto optimal solution, whose formal definition is as follows.
\begin{definition}\label{Definition Pareto}{\bf (Pareto Optimal)} For a multi-objective
optimization problem of the form:
\begin{displaymath}
\min F(x)=[F_1(x),F_2(x),...,F_k(x)]^T
\end{displaymath}
subject to
\begin{displaymath}
 x\in X
\end{displaymath}
a feasible solution \(x^*\in X\), where \(X\) denotes the feasible region, is Pareto
optimal if and only if there does not exist another solution, \(x\in X\),
such that \(F(x)\leq F(x^*)\), and \(F_i(x)<F_i(x^*)\) for at least one function.
\end{definition}
A Pareto optimum requires that no other feasible solutions can improve at
least one objective without deteriorating another. Seeking to attain a Pareto
optimum, we employ  the so-called {\it weighted sum method} \citep{Zadeh, Murata}.
Some new insights on the weighted sum scalarization can be found in \cite{MA}.

\section{Solution Methodology}\label{secPPGM}
The variational inequality \eqref{viconstraint} is a semi-infinite constraint that does not admit known solution schemes. However, it can be reformulated as complementarity constraints as follows:
\begin{equation}
\left(\Psi _{p}(t,h^{\ast
})+\delta_{a,p}\mathcal{Y}_a-\mu_{ij}\right)~\perp~ h^{*}_p(t)  \qquad \forall p\in
\mathcal{P} _{ij},\quad \forall (i,\,j)\in\mathcal{W}
\end{equation}
\begin{equation}
\Psi _{p}(t,h^{\ast})+\delta_{a,p}\mathcal{Y}_a-\mu_{ij}~\geq~ 0 \qquad \forall p\in  \mathcal{P}
_{ij},\quad (i,\,j) \in\mathcal{W}
\end{equation}
\begin{equation}
h^{*}_p~\geq~0\qquad \forall p\in\mathcal{P}
\end{equation}
where $h^{*}=\big(h_p^*:\, p\in\mathcal{P}\big)\in\Lambda$. With complementarity constraints substituting the
VI in the MPEC model, we are able to obtain a single level mathematical program
defined by (\ref{multiobj}) and (\ref{def1})-(\ref{divrt}), and (\ref{emissionvel})-(\ref{totalemission}). Since the complementarity constraints may not satisfy the {\it Mangasarian-Fromovitz constraint qualification} (MFCQ) \citep{RM, IS}, we instead  apply a quadratic penalty method to handle these constraints.   The quadratic penalty approach, also known as the sequential penalty technique, is tested numerically with positive results in solving MPCC problems in \cite{MM}.  Following the quadratic penalty method, we penalize the complementarity constraints and obtain an augmented objective function:
\begin{equation}
\mathcal{U}~=~\left[U_1(h^{*},\mathcal{Y},\mu,M),~~  U_2(h^{*},\mathcal{Y},\mu,M)\right]
\end{equation}
where
\begin{equation}
U_1(h^{*},\mathcal{Y},\mu,M)~=~\sum_{(i,j)\in\mathcal{W}}\sum_{p\in
\mathcal{P}_{ij}}\int^{t_f}_{t_0}\Psi_p(t,h^{*})\, h^{*}_p\, dt+\mathcal{Q}(h^{*},\mathcal{Y},\mu,M)
\end{equation}

\begin{equation}
U_2(h^{*},\mathcal{Y},\mu,M)~=~\sum_{(i,j)\in\mathcal{W}}\sum_{p\in
\mathcal{P}_{ij}}\int^{t_f}_{t_0}
{E}_p(t,h^{*})\,h_p^*\,dt+\mathcal{Q}(h^{*},\mathcal{Y},\mu,M)
\end{equation}
\begin{multline}
\mathcal{Q}(h^{*},\mathcal{Y},\mu,M)~=~M\sum_{(i,j)\in\mathcal{W}}\sum_{p\in
\mathcal{P}_{ij}}\int^{t_f}_{t_0}\left[\left(\Psi_p(t,h^{*})+\delta_{a,p}\mathcal{Y}_{a}-\mu_{ij}\right)h^{*}_p\right]^2dt\\+M\sum_{(i,j)\in\mathcal{W}}\sum_{p\in
\mathcal{P}_{ij}}\int^{t_f}_{t_0}\left[max\left\{\mu_{ij}-\Psi_p(t,h^{*})-\delta_{a,p}\mathcal{Y}_{a},0\right\}\right]^2dt
\end{multline}
where $\mu~=~(\mu_{ij}:ij\in\mathcal{W})$, and \(M\) is a properly large number.
In order to compute the above multi-objective problem, we use a simple but
commonly-used weighted sum scalarization method:

\begin{equation} S_u(h,\mathcal{Y},\mu,M)=\alpha U_{1}(h^{*},\mathcal{Y},\mu,M)+\beta
U_{2}
(h^{*},\mathcal{Y},\mu,M)\end{equation}where $\alpha,\beta \in\Re_{++}$
 are weights for the two objectives. For normalization, we further require that
\(\alpha+\beta=1\). Then, the original MPEC becomes a single-level single-objective
 problem. Such a problem is computed with the gradient projection method
of \cite{DODG}.

\section{Numerical Study}\label{secnumerical}
In this section, we will present a numerical solution of the proposed MPEC problem and demonstrate the  effectiveness of the resulting optimal toll in mitigating both congestion and emission. The toy network of interest is  depicted in Figure \ref{figsmall}, which consists of six arcs and five nodes. There are two origin-destination pairs, $(1,\,3)$ and $(2,\,3)$, among which six paths are utilized, that is,
$$
\mathcal{P}_{1,3}~=~\{p_1,\,p_2,\,p_3,\,p_4\},\qquad \mathcal{P}_{2,3}~=~\{p_5,\,p_6\}
$$
$$
p_1=\{3,\,6\}, \quad p_2=\{1,\,2,\,6\},\quad p_3=\{1,\,2,\,4,\,5\} ,\quad p_4=\{3,\,4,\,5\},\quad p_5=\{6\},\quad p_6=\{4,\,5\}
$$
We assume that  arc $1$ is tollable. Thus the upper-level decision variable of the MPEC problem is a time-varying toll price imposed on arc $1$. The lower level is a Nash-like game whose equilibrium is described by the DUE model where drivers choose their own departure time and route in order to minimize the travel cost, including a toll price. 
\begin{figure}[h!]
   \centering
 \includegraphics[width=.55\textwidth]{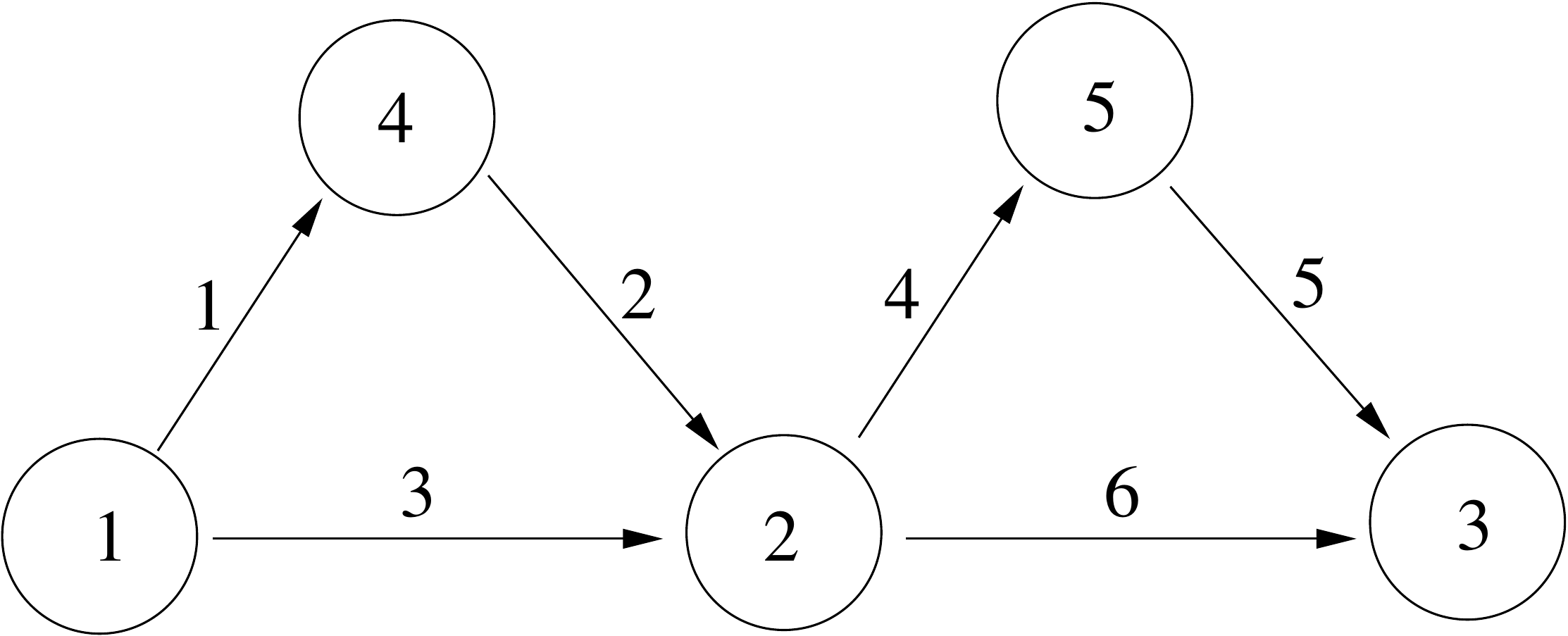}
   \caption{\small The six-arc, five-node network}
   \label{figsmall}
\end{figure}

\subsection{Numerical setup}
We fix a morning commute horizon spanning five hours from 6:00 am to 11:00 am.  The attributes of the arcs are shown in Table \ref{tabarc}.  
\begin{table}[h!]
\begin{center}
\begin{tabular}{|c|c|c|c|}
\hline
Arc  & Jam density    &  Free flow speed   & Length \\
     &   (vehicle/mile)     &  (mile/hour)     &  (mile)   \\
\hline\hline
1   &  400 &35 &  10\\
\hline
2 &  400   &  35 & 10\\
\hline 
3 &   400  &  35  & 10 \\
\hline
4 &  400   & 35 &  20 \\
\hline
5  &  400  &  35 & 20 \\
\hline 
6  & 400   &  35  & 15 \\
\hline
\end{tabular}
\end{center}
\caption{Arc attributes.}
\label{tabarc}
\end{table}
We employ the emission model discussed in Section \ref{secemissionmodel} and Remark \ref{vemission}:
$$
\bar{\mathcal{E}}\big(\bar v(t))~=~\bar v(t)\cdot e_x
$$
where the hot running emission $e_x$ is given by (\ref{emfac}):
\begin{equation}\label{emfac1}
e_x~=~\hbox{BER}\times \exp\left\{ b_1(v-17.03)+b_2(v-17.03)^2 \right\}
\end{equation}
where $\hbox{BER}=2.5$,  $ b_1=-0.04$, $b_2=0.001$ (Smit 2006).  We consider two cases in our computation:
\begin{itemize}
\item[I.] The demand matrix is $(Q_{1,3},\,Q_{2,3})=(820,\,410),$ and the upper
bound for toll price is~$Y_{UB}=10$.

\item[II.] The demand matrix is $(Q_{1,3},\,Q_{2,3})=(1400,\,700)$, and the upper
bound for toll price is~$Y_{UB}=10$.
\end{itemize}

\subsection{Numerical results}
The solution algorithm for MPCC  is implemented  in Matlab (2010a), which runs on the Intel Xeon $3160$ Dual-Core 3.0 GHz processor provided by the Penn State High Performance Computing center. The computational time spent to obtain the numerical solutions below ranges from two to three hours.

\subsubsection{Case I}\label{seccase1}

The numerical results from Case I are displayed  in Figure  \ref{figsmall_due_departure}, \ref{figsmall_mpec_departure}, \ref{figsmall_toll} and \ref{figsmall_diff_departure}. For comparison reasons, we plot simultaneously the equilibrium path flows with and without tolling, in Figure \ref{figsmall_due_departure} and Figure \ref{figsmall_mpec_departure} respectively. The time-varying optimal toll on arc $1$ is depicted in Figure \ref{figsmall_toll}.  

Notice that two paths, $p_2$ and $p_3$,  traverse link $1$; as a result, their associated path flows are affected directly by the toll. In the MPEC solution with toll, the path flows on $p_2$ and $p_3$ diminish to the point where path $p_3$ is not used by any traveler and hardly is path $p_2$. Figure \ref{figsmall_diff_departure} shows the differences in the equilibrium path flows with and without toll. It is clearly observed that most traffic volume on path $p_2,\,p_3$ switch to path $p_1$ and $p_4$, as a consequence of the toll imposed on arc $1$. 

We  compare the two objective functions under the equilibrium conditions with and without toll. The results are summarized in Table \ref{tabcase1}. By imposing the toll, we are able to reduce the total travel cost and total emission by $2.9\%$ and $10.4\%$ respectively.

\begin{figure}[h!]
\begin{minipage}[b]{.49\textwidth}
\centering
\includegraphics[width=.8\textwidth]{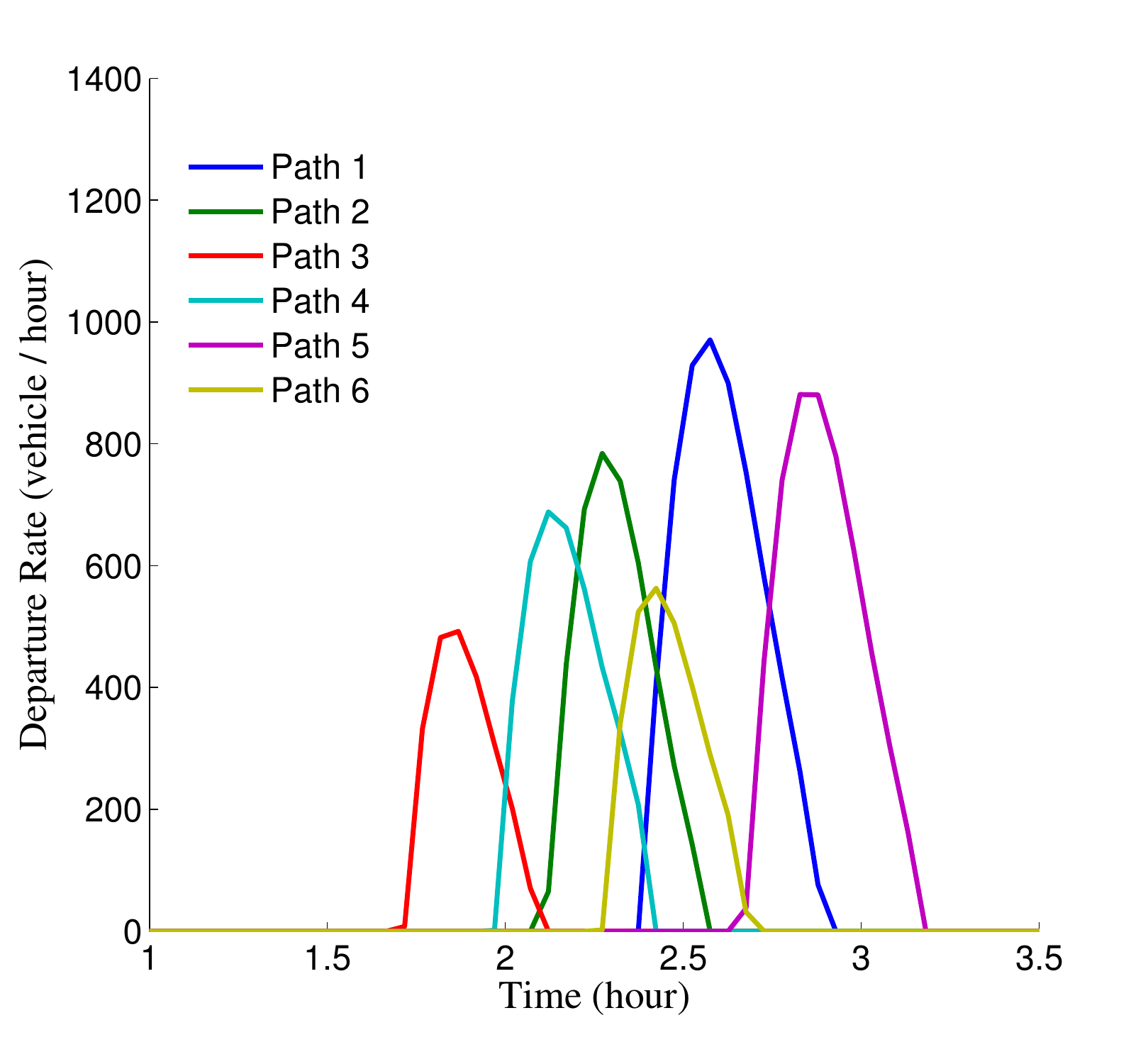}
\caption{Case I: DUE solution without any toll}
\label{figsmall_due_departure}
\end{minipage}
\begin{minipage}[b]{.49\textwidth}
\centering
\includegraphics[width=.8\textwidth]{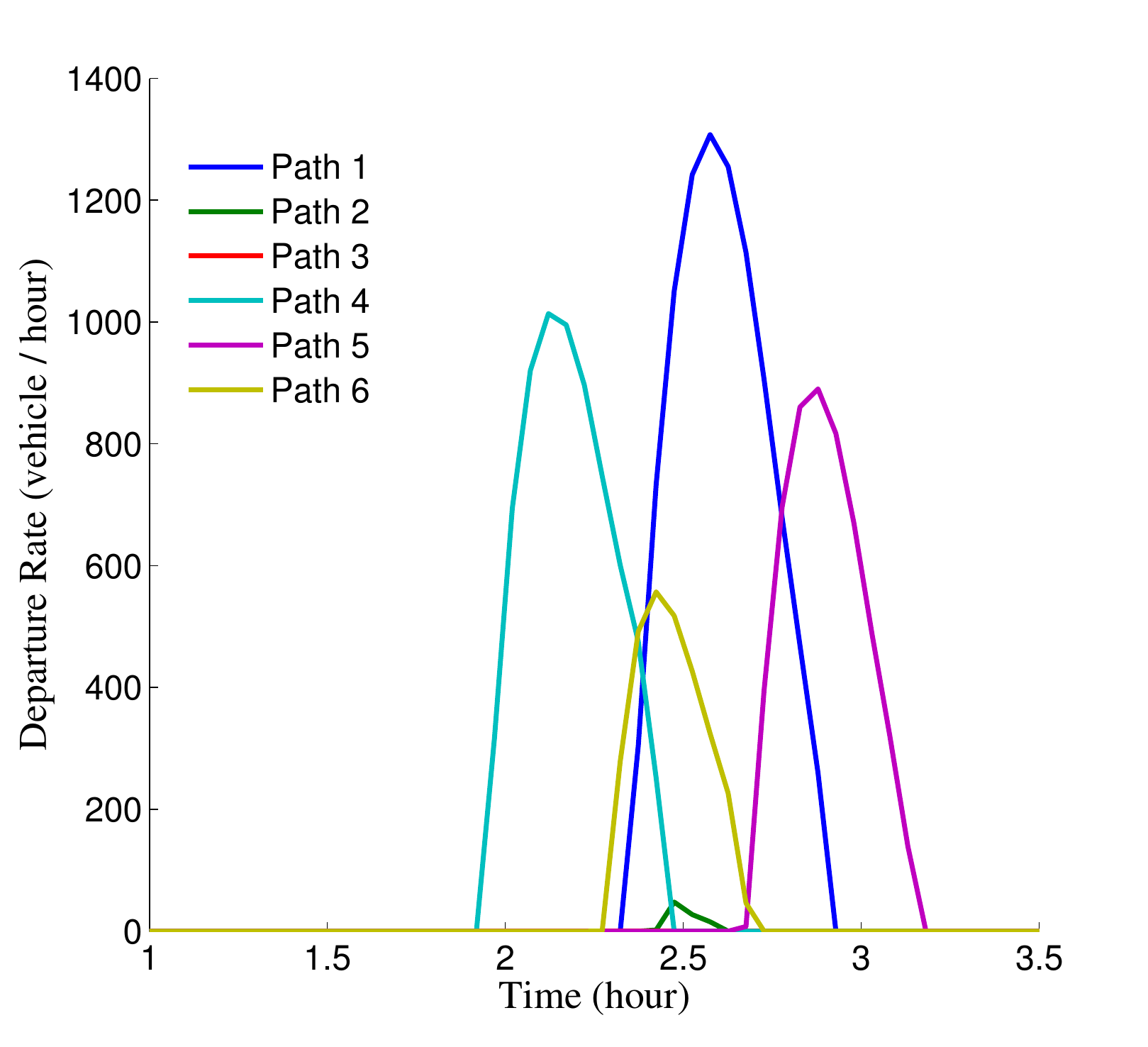}
\caption{Case I: DUE solution with optimal toll.}
\label{figsmall_mpec_departure}
\end{minipage}
\end{figure}

\begin{figure}[h!]
\begin{minipage}[b]{.49\textwidth}
\centering
\includegraphics[width=.8\textwidth]{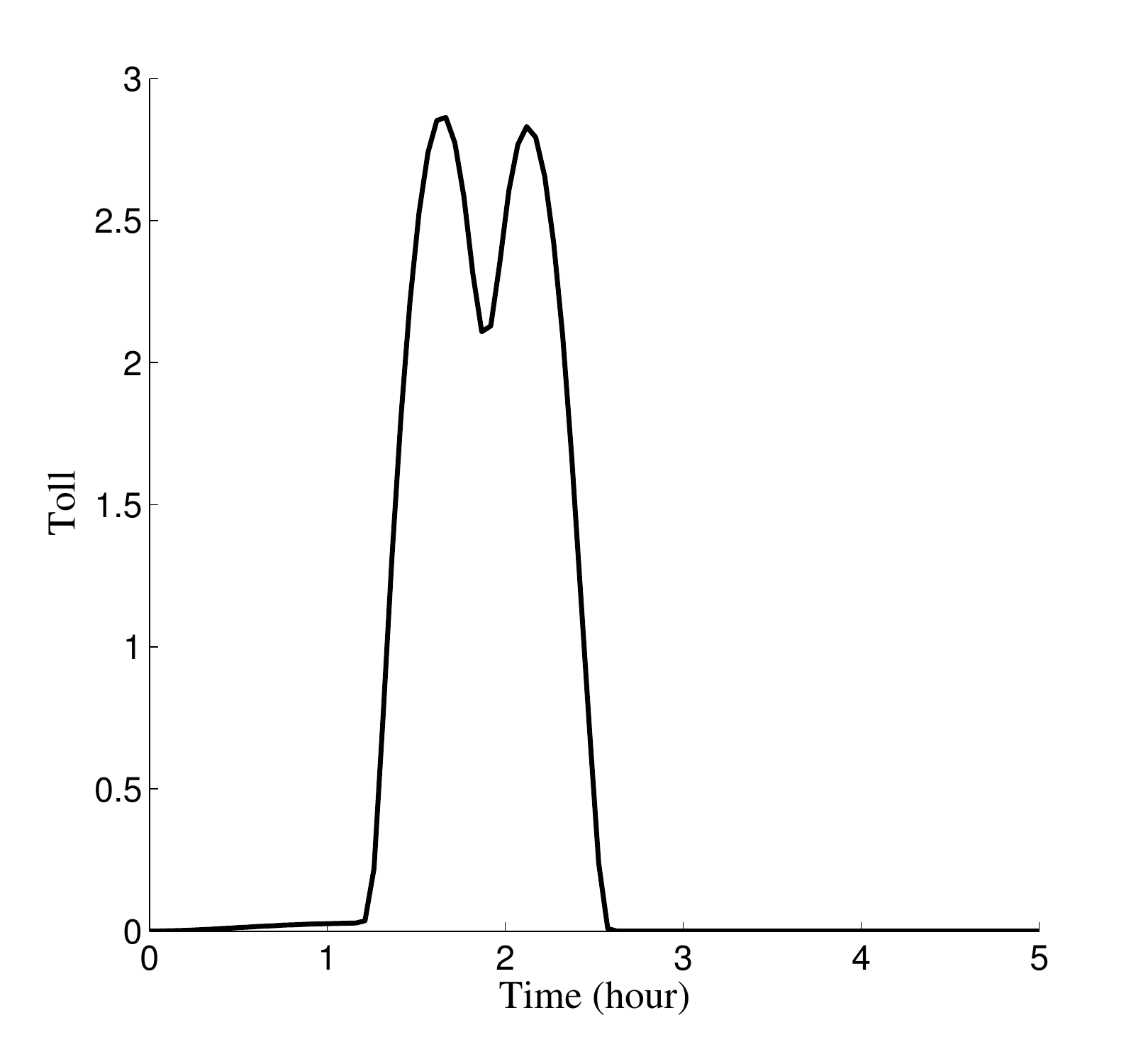}
\caption{Case I: optimal toll on arc 1.}
\label{figsmall_toll}
\end{minipage}
\begin{minipage}[b]{.49\textwidth}
\centering
\includegraphics[width=.8\textwidth]{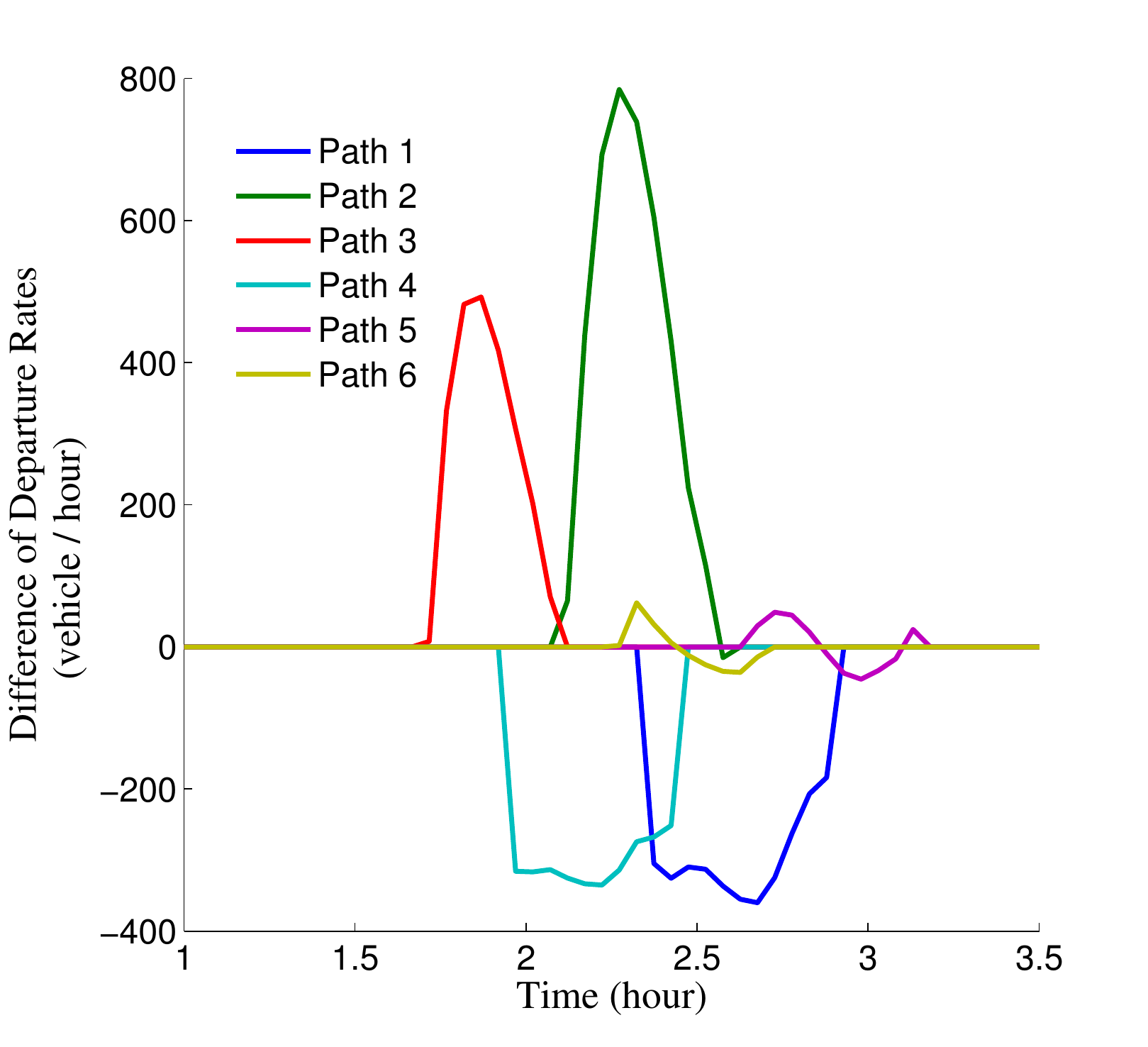}
\caption{Case I: differences in equilibrium path flows with and without toll.}
\label{figsmall_diff_departure}
\end{minipage}
\end{figure}

\begin{table}[h!]
\begin{center}
\begin{tabular}{|c|c|c|}
\hline
                                   & Total travel cost    &  Total emission     \\
\hline
DUE without toll       &   3.4744E+04        &      3.1789E+06  \\
\hline
DUE with toll            &  3.3723E+04     &  2.8483E+06         \\
\hline 
\end{tabular}
\end{center}
\caption{Case I: comparison of objective functions under equilibrium flow. }
\label{tabcase1}
\end{table}

\subsubsection{Case II}
In Case II, the travel demand between each O-D pair is significantly increased. The numerical solutions are shown in  Figure \ref{figlarge_due_departure}, \ref{figlarge_mpec_departure}, \ref{figlarge_toll} and \ref{figlarge_diff_departure}, which displays the same quantities as in Case I. Unlike the first case,  Case II shows only minor change of the DUE path flows with and without toll. We interpret such results with the following intuition: when the demand increases, the system becomes less sensitive to control parameters, making the system less controllable. This is also reflected from the comparison of objectives, as shown in Table \ref{tabcase2}. The reduction of total travel cost and total emission is only $0.04\%$ and $0.45\%$.

\begin{figure}[h!]
\begin{minipage}[b]{.49\textwidth}
\centering
\includegraphics[width=.8\textwidth]{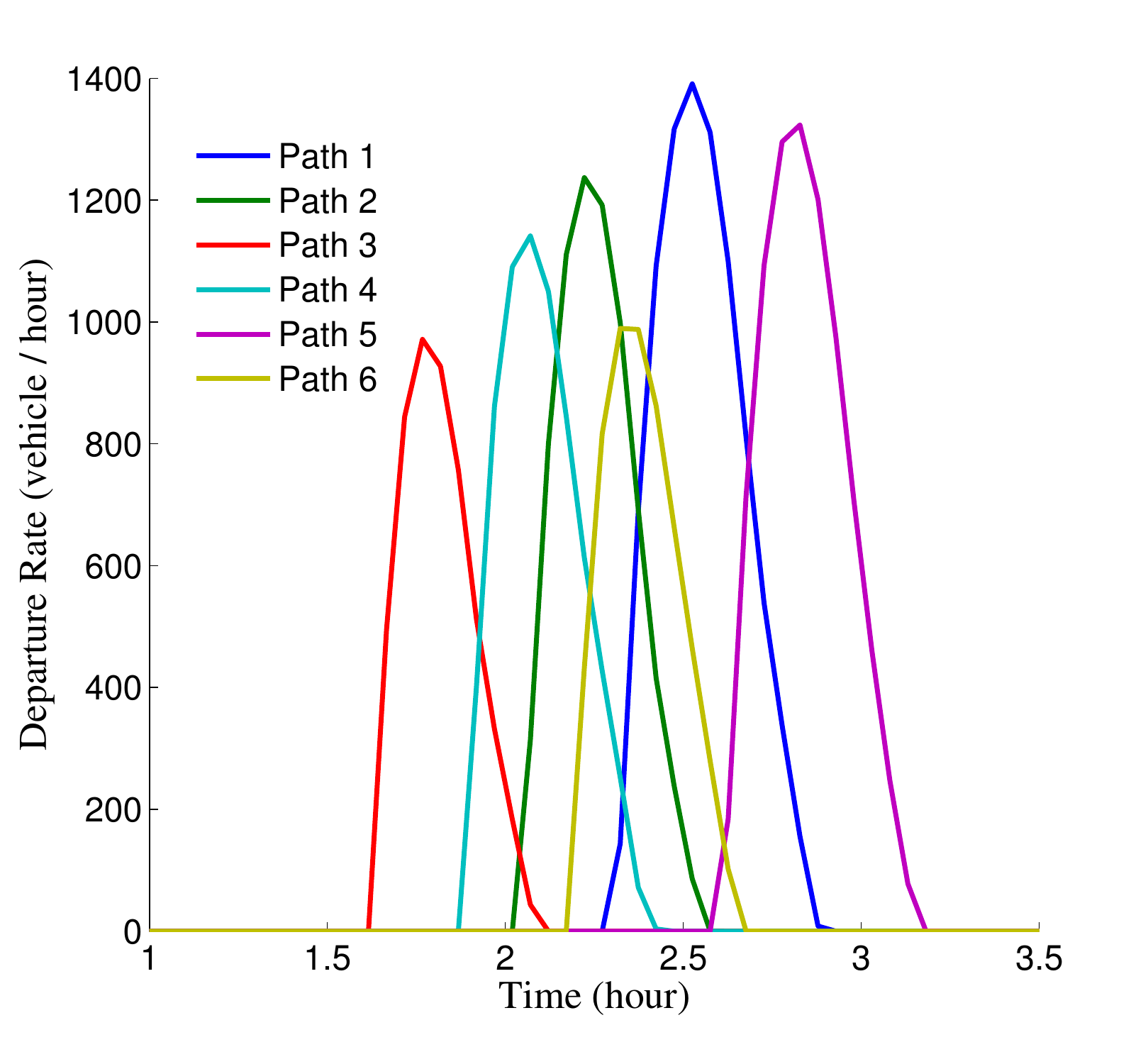}
\caption{Case II: DUE solution without any toll.}
\label{figlarge_due_departure}
\end{minipage}
\begin{minipage}[b]{.49\textwidth}
\centering
\includegraphics[width=.8\textwidth]{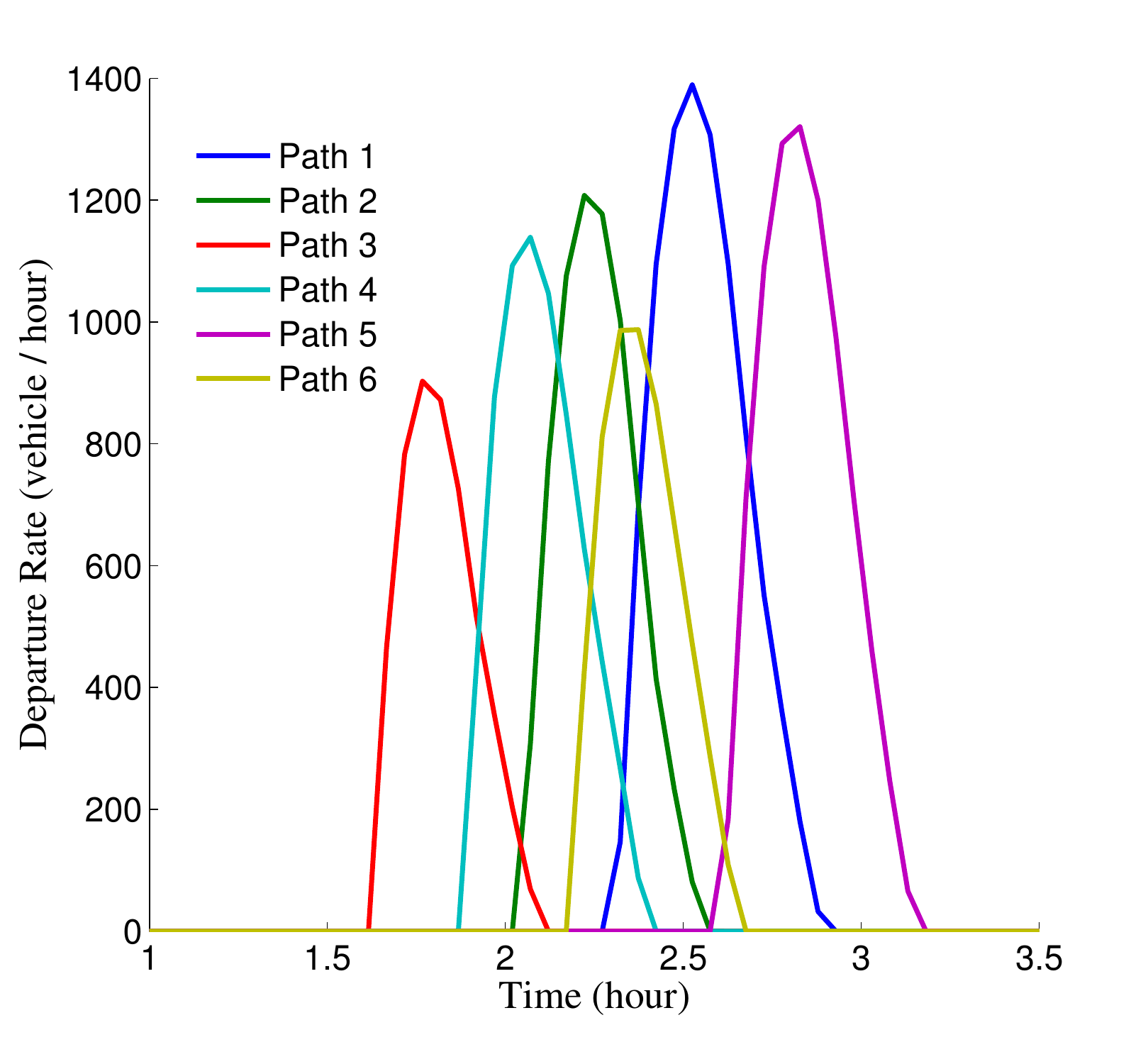}
\caption{Case II: DUE solution with optimal toll.}
\label{figlarge_mpec_departure}
\end{minipage}
\end{figure}

\begin{figure}[h!]
\begin{minipage}[b]{.49\textwidth}
\centering
\includegraphics[width=.8\textwidth]{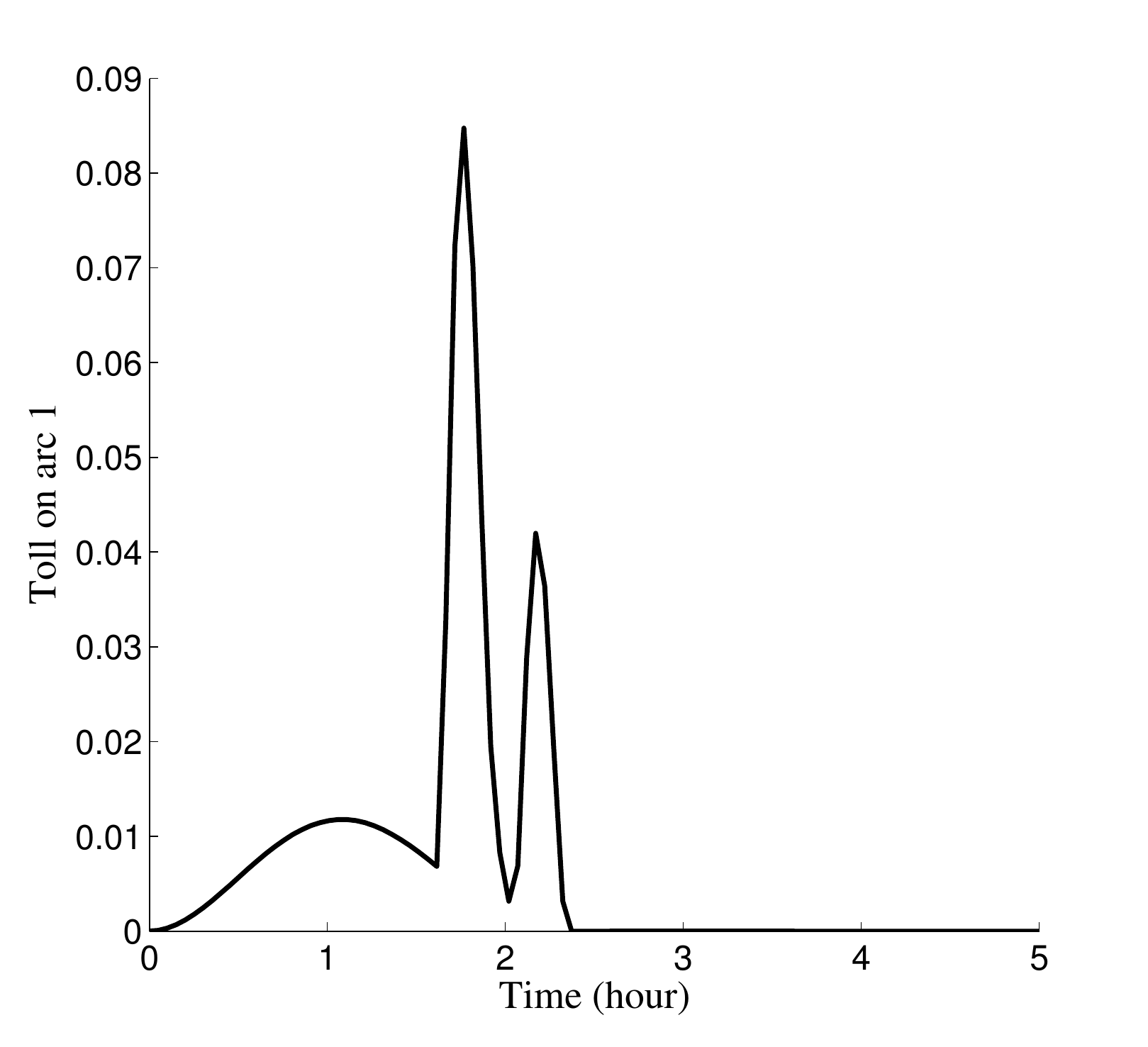}
\caption{Case II: optimal toll on arc 1.}
\label{figlarge_toll}
\end{minipage}
\begin{minipage}[b]{.49\textwidth}
\centering
\includegraphics[width=.8\textwidth]{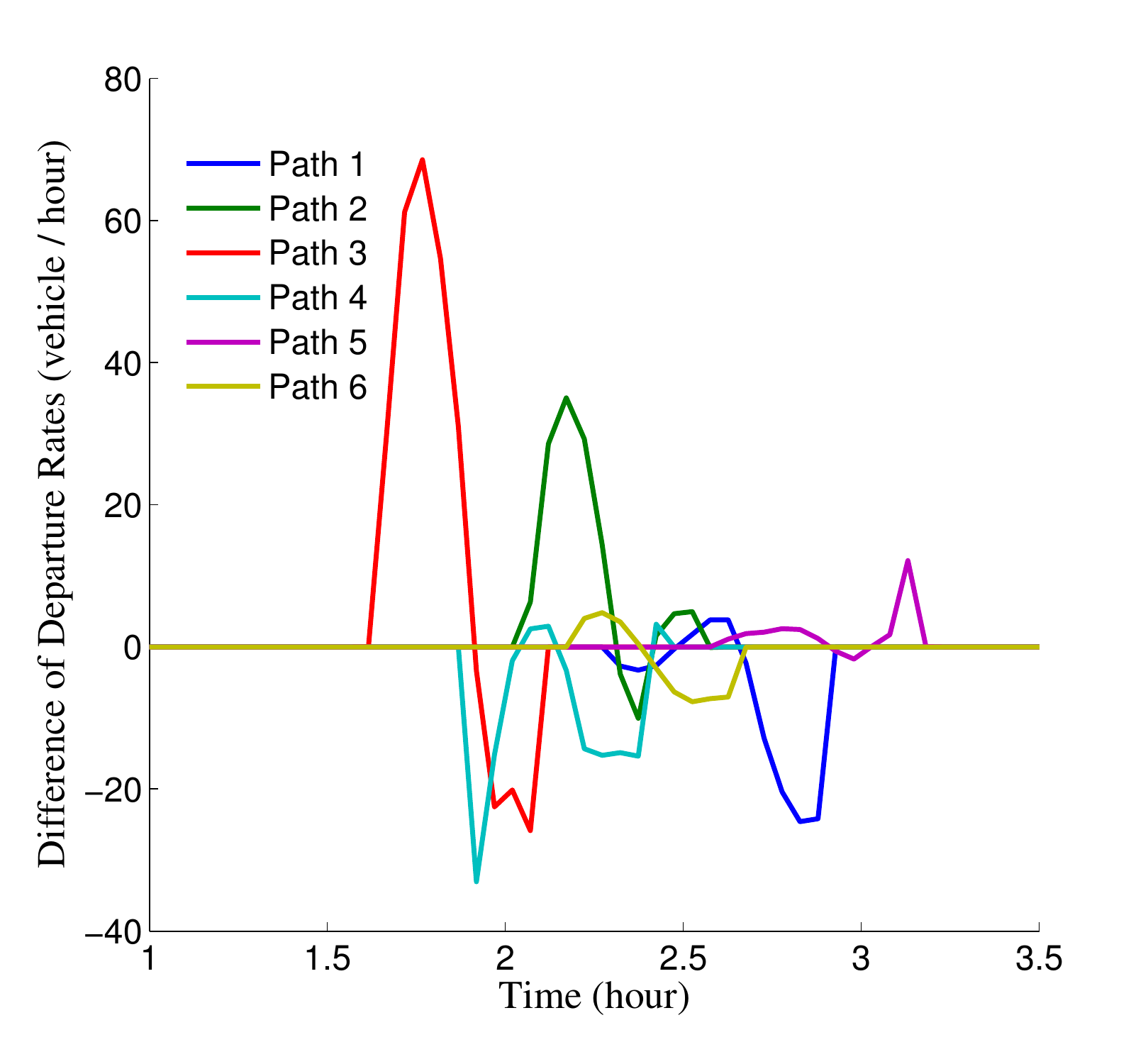}
\caption{Case II: differences of path flows between DUE without toll and DUE with toll.}
\label{figlarge_diff_departure}
\end{minipage}
\end{figure}

\begin{table}[h!]
\begin{center}
\begin{tabular}{|c|c|c|}
\hline
                                   & Total travel cost    &  Total emission     \\
\hline
DUE without toll      &   7.5962E+04         &   5.4119E+06    \\
\hline
DUE with toll            &   7.5932E+04         &   5.3878E+06         \\
\hline 
\end{tabular}
\end{center}
\caption{Case II: comparison of objective functions under equilibrium flow.}
\label{tabcase2}
\end{table}

\subsection{Different weights}
The multi-objective program is solved using the weighted sum scalarization method. We are interested to find out how the solution is affected by using different weights for the total effective delay and the total emission. Such test is conducted for both Case I and Case II,  with  results  summarized in Table \ref{tabdiffcase1} and Table \ref{tabdiffcase2}, respectively. We indicate by $\alpha$ the weight for the effective delay, and by $\beta$ the weight for the emission. 

\begin{table}[h]
\begin{center}
\begin{tabular}{|c|c|c|c|c|}
\hline
                               & Total travel cost    &  Total emission   & $\alpha$  & $\beta$    \\
\hline
Weight i,             &   3.3723E+04         &   2.8482E+06         &    0.0988   &  0.9011 \\
\hline
Weight ii,              &   3.3723E+04         &   2.8483E+06         &    0.9434   &  0.0566  \\
\hline 
\end{tabular}
\end{center}
\caption{Case I: comparison of objectives for different choices of weights. $\alpha$ is the weight of total travel cost, $\beta$ is the weight of total emission.}
\label{tabdiffcase1}
\end{table}

\begin{table}[h!]
\begin{center}
\begin{tabular}{|c|c|c|c|c|}
\hline
                               & Total travel cost    &  Total emission   & $\alpha$  & $\beta$    \\
\hline
Weight i,             &   7.5932E+04         &   5.3878E+06         &    0.0138   &  0.9862 \\
\hline
Weight ii,            &   7.8360E+04         &   5.2396E+06       &    0            &      1  \\
\hline 
Weight iii,           &   7.5858E+04         &   5.4175E+06       &     1         &     0  \\
\hline 
\end{tabular}
\end{center}
\caption{Case II: comparison of objectives for different choices of weights. $\alpha$ is the weight of total travel cost, $\beta$ is the weight of total emission.}
\label{tabdiffcase2}
\end{table}

\subsection{Discussions}
In the numerical example presented above, the tolling problem with multiple objectives on a toy network is solved. The results display certain interesting phenomena of the proposed model and provide insights to the sustainable management of road congestion in general.

As our first observation, a Braess-type paradox is created in Case I in Section \ref{seccase1}.  Namely,  the performance of the network, whether in terms of minimizing effective travel delay or in terms of minimizing emission, is enhanced in the more constrained system (the one with toll).  In general, the Braess paradox \citep{Braess}  tells us that, a network performance enhancement that is local (in space and, by implication, also in time) may produce global performance degradation. The Braess paradox is a phenomenon widely accepted as a fundamental feature of a large class of networks, namely those with noncooporative users and flow dependent costs (delays or latencies).  In our model, the travelers are assumed to be Nash agents who seek to minimize their own disutility. By imposing a nontrivial  toll on arc 1, affordable time windows on path \(p_2\) and path \(p_3\) become significantly
smaller. In particular, we notice that in the presence of the toll, path $p_3$ is
completely abandoned by the Nash agents.   With fewer affordable choices on path and/or departure time, the system capacity becomes more constrained.  However,  both the total travel cost and the emission amount are reduced; in other words,  a higher efficiency of the traffic network is attained in terms of transportation efficiency and environmental sustainability. To the best of our knowledge, this is the first observation of the Braess paradox in the context of environmental sustainability.

Second, tolls can act as effective stimuli in a transportation system.
We observe from Table 2 (in Case I) and Table 3 (in Case II) that by properly choosing
the toll prices, one can reduce both the total traffic cost and total emission. 
In Case I, the price of toll  on the first arc of path \(p_2\)  renders this
path no longer an affordable choice, i.e. the equilibrium flow on this path
vanishes, which, nonetheless,  creates a Braess-like paradox as discussed
above.  

Thirdly, the objectives of minimizing the total traffic cost and the total emission are neither completely conflicting  nor totally aligned with each other. We can tell from Table 3 that the Pareto optimal solutions can provide improvement on both criteria in comparison with the equilibrium state in anarchy as shown in Table 2 and 3, while we also observe in Table  5 that improving one objective might compromise  the other objective. Nonetheless, the latter observation is case-specific; as we compare Table 4 and 5, the tradeoffs between the objectives are less significant in  Table 4 than in Table 5.

Finally, by comparing Table 4 and Table 5, in response to changing weights
for the two objectives in our weighted sum approach, Case I (in Table 4)
shows  very minor changes in objective values compared to Case II (in Table
5). The reason for such a difference in the sensitivity to weight perturbations is still of our
research curiosity. Nonetheless, this points to a necessity of a careful
determination of weights for the two objectives in our model in order to attain the optimization goal of the central control.

\section{Conclusion}
This paper proposes a congestion pricing problem that takes into account the environmental impact of traffic dynamics on a vehicular network. The optimal tolling problem is formulated as a bi-level problem where the upper level decision maker (central authorigy) seek to simultaneously minimize both congestion and vehicle-driven environmental deterioration, while the lower-level decision makers (travelers)  engage in a Nash-like game by selfishly minimizing their own travel delay and/or arrival penalties. The lower-level model is a dynamic user equilibrium (DUE) which is expressed as a differential variational inequality. A mathematical program with equilibrium constraints (MPEC) formulation of the bi-level problem is presented, which is then reformulated as a mathematical program with complementarity constraints (MPCC). In order to avoid violation of constraint qualifications, we  apply a quadratic penalty-based method to the MPCC.  The relaxed program is solved with the  gradient projection algorithm presented in \cite{DODG}, with the two objectives handled via a weighted-sum scalarization.

The lower-level DUE problem is solved with a fixed-point algorithm in Hilbert space \citep{FrieszMookherjee, Friesz2011, Friesz2013}. Such an algorithm requires constant evaluation of effective path delays with established path flows, which is recognized as the dynamic network loading (DNL) procedure. By nature of our proposed model, an emission estimation procedure needs to be embedded in the DNL subproblem; this is done in this paper by employing a speed-related emission function in \cite{emfac} and by integrating such function with path delays produced by the DNL procedure.

The numerical example demonstrates the effectiveness of congestion toll in controlling and reducing both total travel cost and emission. We also report a Braess-type paradox where a more constrained system results in higher transportation efficiency and less environmental deterioration.

\cite{Wismans} employ genetic algorithms and response surface methods to approximate the whole set of the Pareto frontier and used pruning method to facilitate decision making. Also of our research curiosity is the generation of the Pareto frontier  such that the trade-offs between the two objective functions can be further investigated. \cite{KV} point out that a weighted sum scalarization approach could obtain only convex part of the Pareto frontier, which might lead
to selecting suboptimal trajectories. To resolve such an issue, they provide
an alternative `marching' method. The application of their approach to our
model is also of our future research interest.










\begin{thebibliography}{00}



\bibitem[Anitescu, 2000]{Anitescu} Anitescu, M. (2000). On solving mathematical
programs with complementarity constraints as nonlinear programs. \emph{Preprint
ANL/MCS-P864-1200}, MCS Division, Argonne National Laboratory, Argonne, IL.


\bibitem[Arnott et al., 1990]{ADL} Arnott, A., de Palma, A., \& Lindsey, R. (1990). Departure time and route choice for the morning commute. \emph{Transportation Research Part B}, 24(3), 209-228.

\bibitem[Arnott and Kraus, 1998]{AK} Arnott, R., \& Kraus, M.
(1998). When are anonymous congestion charges consistent with marginal cost pricing?
\emph{Journal of Public Economics}, 67(1), 45-64.

\bibitem[Arnott and Small, 1994] {AS} Arnott, R., \& Small, K.
(1994). The economics of traffic congestion. \emph{American Scientist}, 82(5), 446-455.


 \bibitem[Ban et al., 2006]{Ban} Ban, J.X., Liu, H.X., Ferris, M.C., \&
Ran, B. (2006). A general MPCC model and its solution algorithm for continuous network design problem. \emph{Mathematical and Computer Modeling}, 43, 493-505.
 
 
\bibitem[Barth et al., 1996]{Barth} Barth, M., An, F., Norbeck, J., Ross, M. (1996). Modal emission modeling: A physical approach. \emph{Transportation Research Record}, 1520, 81-88.
 
\bibitem[Braess, 1969]{Braess} Braess, D. (1969). \"Uber ein Paradoxon aus
der Verkehrsplanung. \emph{Unternehmensforschung}, 12, 258-268.

\bibitem[Braid, 1996]{Braid} Braid, R. (1996). Peak-load pricing of
a transportation route with an unpriced substitute. \emph{Journal of Urban Economics}, 
40(2), 179-197.

\bibitem[Bressan and Han, 2011]{BH} Bressan, A.,  \& Han, K. (2011). Optima and Equilibria for a model of traffic flow. \emph{SIAM Journal on Mathematical Analysis},  43(5), 2384-2417. 

\bibitem[Bressan and Han, 2012]{BH1} Bressan, A., \&  Han, K. (2011). Nash Equilibria for a Model of Traffic Flow with Several Groups of Drivers.  \emph{ESAIM: Control, Optimization and Calculus of Variations}, 18(4), 969-986.



\bibitem[CARB, 2000]{emfac} CARB (2000). Public Meeting to Consider Approval
of Revisions to the State's On-Road Motor Vehicle Emissions Inventory – Technical
Support Document. California Air Resources Board. May 2000.

\bibitem[Colombo and Groli, 2003]{Colombo and Groli} Colombo, R.M., Groli, A., 2003. Minimizing stop and go waves to optimise traffic flow. \emph{Applied Mathematics Letters}, 17, 697-701. 



\bibitem[Daganzo, 1994]{CTM1} 
Daganzo, C.F. (1994). The cell transmission model. Part I:  A simple dynamic representation of highway traffic. \emph{Transportation Research Part B}, 28(4), 269-287.


\bibitem[Daganzo, 1995]{CTM2}
Daganzo, C.F. (1995). The cell transmission model. Part II: Network traffic. \emph{Transportation Research Part B}, 29(2), 79-93.





\bibitem[De Palma and Lindsey, 2000]{DL} De Palma,
A., \& Lindsey, R. (2000). Private toll road: Competition under various ownership
regimes. \emph{The Annals of Regional Sciences}, 34(1), 13-35.


\bibitem[Dial, 1999]{Dial1999} Dial, R. (1999). Minimum revenue congestion pricing
Part I: A fast algorithm for the single-origin case. \emph{Transportation Research Part B}, 33(3), 189-202.

\bibitem[Dial, 2000]{Dial2000} Dial, R. (2000). Minimum revenue congestion pricing
Part II: A fast algorithm for the general case. \emph{Transportation Research Part B} 34(8), 645-665.


\bibitem[Ekstr\"om et al., 2004]{Ekstrom2004} Ekstr\"om, M., Sjodin, A., \& Andreasson, K. (2004). Evaluation of the COPERT III emission model within on-road. \emph{Atmospheric Environment}, 38, 6631-6641. 


\bibitem[Evans, 2010]{Evans}
Evans, L.C. (2010).  \emph{Partial Differential Equations}.  Second edition. 
American Mathematical Society, Providence, RI.



\bibitem[Friesz, 2010]{DODG} Friesz, T.L. (2010). \emph{Dynamic Optimization and Differential Games}. Springer, New York.



\bibitem[Friesz et al., 1993]{Friesz1993} Friesz, T.L.,  Bernstein, D.,  Smith, T.,  Tobin, R., \& Wie, B. (1993). A variational inequality formulation of the dynamic network user equilibrium problem. \emph{Operations Research}, 41(1), 80-91.



\bibitem[Friesz et al., 2001]{Friesz2001} Friesz, T.L., Bernstein, D., Suo, Z., \& Tobin, R.L. (2001). Dynamic network user equilibrium with state-dependent time lags. \emph{Networks and Spatial Economics}, 1(3-4), 319-347.


\bibitem[Friesz and Mookherjee, 2006]{FrieszMookherjee} Friesz, T.L., \& Mookherjee, R. (2006). Solving the dynamic network user equilibrium problem with state-dependent time shifts. \emph{Transportation Research Part B}, 40(3), 207-229. 





\bibitem[Friesz et al., 2007]{Friesz2007} Friesz, T.L., Kwon, C.,  \& Mookherjee, R. (2007). A computable theory of dynamic congestion pricing. \emph{Transportation and Traffic Theory 2007}, 1-26. 



\bibitem[Friesz et al., 2011]{Friesz2011} Friesz, T.L., Kim,  T., Kwon, C., \&  Rigdon, M.A. (2011). Approximate network loading and dual-time-scale dynamic user equilibrium.  \emph{Transportation Research Part B}, 45(1), 176-207. 



\bibitem[Friesz et al., 2013]{Friesz2013} Friesz, T.L., Han, K.,  Neto, P.A., Meimand, A., \& Yao, T. (2013). Dynamic user equilibrium based on a hydrodynamic model. \emph{Transportation Research Part B}, 47(1), 102-126.



\bibitem[Hearn and Ramana, 1998]{HR} Hearn, D.W., \& Ramana,
M.V. (1998). Solving congestion toll pricing models. \emph{In: Equilibrium and Advanced
Transportation Modeling}, P. Marcotte, S. Nguyen (eds.), Kluwer Academic Publishers,
Boston, pp. 109-124.


\bibitem[Izmailov and Solodov, 2004]{IS} ]Izmailov, A. F., \& Solodov, M. V. (2004). Newton-type methods for optimiztion problems without constraint qualifications. \emph{SIAM Journal on Optimization}, 15(1), 210-228.

\bibitem[Kent and Mudford, 1979]{KM} Kent, J.H. \& Mudford,
N.R. (1979). Motor vehicle emissions and fuel consumption modeling. \emph{Transportation
Research Part A}, 13(6), 395-406.

\bibitem[Kumar and Vladimirsky, 2010]{KV} Kumar, A., \& Vladimirsky,
A. (2010). An efficient method for multi-objective optimal control and optimal
control subject to integral constraints. \emph{Journal of Computational Mathematics},
28(4), 517-551.


\bibitem[Lawphongpanich and Hearn, 2004]{LH} Lawphongpanich,
S., \& Hearn, D. (2004). An MPEC approach to second-best toll pricing. \emph{Mathematical
Programming}, 101(1), 33-55.


\bibitem[Lax, 1957]{Lax} Lax, P.D. (1957). 
Hyperbolic systems of conservation laws II.
\emph{Communications on Pure and  Applied Mathematics}, 10(4), 537Ð566. 



\bibitem[Lighthill and Whitham, 1955]{LW}
Lighthill, M.,  \&  Whitham, G. (1955). On kinematic waves. II. A theory of traffic flow on long crowded roads.  \emph{Proceedings of the Royal Society of London: Series A}, 
 229, 317- 345.


 
 
 
\bibitem[Marler and Arora, 2010]{MA} Marler, R.T, \& Arora, J.S. (2010). The
weighted sum method for multi-objective optimization: new insights. \emph{Structural and
Multidisciplinary Optimization}, 41(6), 853-862.

\bibitem[Monteiro and Meira, 2011]{MM} Monteiro, M.T.T., \& Meira, J.F.P. (2011).
A penalty method and a regularization strategy to solve MPCC. \emph{International Journal
of Computer Mathematics}, 88(1), 145-149. 



\bibitem[Murata et al., 1996]{Murata} Murata, T., Ishibuchi, H., \& Tanaka, H. (1996).
Multi-objective genetic algorithm and its applications to flowshop scheduling.
\emph{Computers and Industrial Engineering}, 30(4), 957-968.
 
 
 \bibitem[Nagurney, 1993]{Nagurney} Nagurney, A. (1993). \emph{Network economics: A variational inequality approach}. Kluwer Academic Publishers, Norwell, Massachusetts. 


\bibitem[Panis et al., 2006]{Panis} Panis, L.I., Broekx, S., \&  Liu, R. (2006). Modeling instantaneous traffic emission and the influence of traffic speed limits. \emph{Science of the Total Environment}, 371, 270-285. 


\bibitem[Pareto, 1906]{Pareto} Pareto, V. (1906). \textit{Manuale
di Economica Politica, Societa Editrice Libraria}. Milan; translated into
English by A.S. Schwier as \textit{Manual of Political Economy}, edited by
A.S. Schwier and A.N. Page, 1971. New York: A.M. Kelly.


\bibitem[Pigou, 1920]{Pigou} Pigou, A. (1920). \emph{The economics of welfare}. London: Macmillan and Co.

\bibitem[Raghunathan et al., 2004]{Raghunathan} Raghunathan, A.U., Diaz, M.S.,\& Biegler, L.T.  (2004). An MPEC formulation for dynamic optimization of distillation operations. \emph{Computers and Chemical Engineering}, 28, 2037-2052.


\bibitem[Rakha et al., 2004]{Rakha} Rakha, H., Ahn, K., \& Trani, A. (2004). Development of VT-Micro model for estimation hot stabilized light duty vehicle and truck emissions. \emph{Transportation Research Part D}, 9(1), 49-74. 


\bibitem[Ralph and Wright, 2004]{RW} Ralph, D., \& Wright, S.J. (2004). Some properties of regularization and penalization schemes for MPECs. \emph{Optimization Methods and Software}, 19, 527-556.


\bibitem[Richards, 1956]{Richards} Richards, P.I., 1956. Shockwaves on the highway. \emph{Operations Research},  4(1), 42-51.



\bibitem[Rodrigues and Monteiro, 2006]{RM} Rodrigues, H.S., \&  Monteiro, M.T.T. (2006). Solving mathematical programs with complementarity constraints with nonlinear solvers. \emph{Recent Advances in Optimization. Lecture Notes in Economics and Mathematical Systems,} 563(IV), 415-424.

\bibitem[Rose et al., 1965]{Rose} Rose, A.H., Smith, R., McMichael,
W.F. \& Kruse, R.F. (1965) Comparison of auto
exhaust emissions in two major cities. \emph{Journal of the Air Pollution Control
Association}, 15(8), 362-371.



\bibitem[Smit, 2006]{Smit2006} Smit, R. (2006). An examination of congestion in road traffic emission models and their application to urban road networks. Ph.D. thesis, Griffith University. 


\bibitem[Wismans, 2012]{Wismans} Wismans, L. (2012). Towards sustainable dynamic traffic management. Ph.D. Dissertation. University of Twente, Faculty of Engineering Technology, Center for Transport Studies.



\bibitem[Yang and Huang, 2005] {YH} Yang, H., \& Huang, H.J. 2005.
\textit{Mathematical and Economic Theory of Road Pricing}. Elsevier Oxford.



\bibitem[Yao et al., 2012]{Yao} Yao, T., Friesz, T.L., Chung, B.D.,  Liu, H. 2012. Dynamic congestion pricing under uncertainty: a robust optimization approach. \emph{Transportation Research Part B}, 46(10), 1504-1518.


\bibitem[Zadeh, 1963]{Zadeh} Zadeh, L.A. (1963). Optimality and non-scalar-valued
performance criteria. \emph{IEEE Transactions on Automatic Control} AC-8, 59–60




\bibitem[Zegeye et al., 2010]{Zegeye et al} Zegeye, S.K.,  Schutter, B.,   Hellendoorn, J.,  Breunesse, E.A. (2010). Integrated macroscopic traffic flow and emission model based on METANET and VT- micro. \emph{Proceedings of the International Conference on Models and Technologies for Intelligent Transportation Systems}, Rome, Italy, June 2009 (G. Fusco, ed.), Rome, Italy: Aracne Editrice, ISBN 978-88-548-3025-7,  86Ð89.



\end{thebibliography}



\end{document}